\documentclass[letter,11pt,reqno,oneside]{amsart}

\usepackage[left=1in,right=1.5in,top=1in,bottom=1in]{geometry}
\usepackage{amsmath,amsfonts,amsthm,mathrsfs,amssymb,bm,cite}
\usepackage{graphicx,epstopdf,booktabs,array}
\usepackage{paralist,verbatim,caption,subcaption}
\usepackage[usenames]{color}

\newtheorem{thm}{Theorem}[section]

\theoremstyle{definition}
\newtheorem{defn}{Definition}[section]
\theoremstyle{remark}

\numberwithin{equation}{section}

\newcommand{\beq}{\begin{equation}}
\newcommand{\eeq}{\end{equation}}

\title[Reconstruction via interior transmission eigenfunctions]{Reconstruction via the intrinsic geometric structures of interior transmission eigenfunctions}

\author{Jingzhi Li}
\address{Department of Mathematics, Southern University of Science and Technology, Shenzhen, China.\vspace*{-4mm}}
\address{\vspace*{-4mm}}
\email{li.jz@sustc.edu.cn}

\author{Xiaofei Li}
\address{Department of Mathematics, Inha University, Incheon, South Korea.\vspace*{-4mm}}
\address{\vspace*{-4mm}}
\email{xiaofeilee@hotmail.com}

\author{Hongyu Liu}
\address{Department of Mathematics, Hong Kong Baptist University, Kowloon Tong, Hong Kong SAR.\vspace*{-4mm}}
\address{\vspace*{-4mm}and}
\address{HKBU Institute of Research and Continuing Education, Virtual University Park, Shenzhen, P. R. China.}
\email{hongyu.liuip@gmail.com}

\date{}

\begin{document}
	
	\maketitle
	
	\begin{abstract}
		We are concerned with the inverse scattering problem of extracting the geometric structures of an unknown/inaccessible inhomogeneous medium by using the corresponding acoustic far-field measurement. Using the intrinsic geometric properties of the so-called interior transmission eigenfunctions, we develop a novel inverse scattering scheme. The proposed method can efficiently capture the cusp singularities of the support of the inhomogeneous medium. If further a priori information is available on the support of the medium, say, it is a convex polyhedron, then one can actually recover its shape. Both theoretical analysis and numerical experiments are provided. Our reconstruction method is new to the literature and opens up a new direction in the study of inverse scattering problems.
	\end{abstract}
	
	

	\section{Introduction}
	This work is concerned with the inverse medium scattering problem of extracting the geometric structures of a penetrable scatter by the corresponding acoustic wave detection. Let $D\subset\mathbb{R}^l,l=2,3$, be a bounded Lipschitz domain with a connected complement $D^c:=\mathbb{R}^l\backslash\bar{D}$. Assume that the refractive index of the background space $\mathbb{R}^l\backslash\bar{D}$ is $n(x)=1$, and the refractive index of $D$ is $n(x)=\frac{c_0^2}{c^2(x)}\neq 1$, where $c_0$ is the sound speed in $\mathbb{R}^l\backslash\bar{D}$, and $c$ is the sound speed in $D$. We assume that the refractive index is always positive, and let $q(x):=n(x)-1$.
	
	Given an incident field $u^i$, the presence of the scatter $D$ will give rise to a scattered field $u^s$. We take $u^i(x) = e^{ikx\cdot d}$ to be a time-harmonic plane wave, where $i=\sqrt{-1}$, $d\in\mathbb{R}^{l-1}$ is the incident direction, and  $k\in\mathbb{R}_+$ is the wave number. We define $u(x)=u^i(x)+u^s(x)$ to be the total field, which satisfies the following Helmholtz system
	\begin{equation}
	\begin{cases}
	\Delta u+k^2(1+q)u=0~\quad &~\mbox{in}~ \mathbb{R}^l,\\
	u=u^i+u^s~\quad &~\mbox{in}~ \mathbb{R}^l,\\
	\lim_{r\rightarrow\infty}r^{(l-1)/2}(\frac{\partial u^s}{\partial r}-\mathrm{i}ku^s)=0,~&~r=|x|.
	\end{cases}
	\label{set}
	\end{equation}
	The last limit in \eqref{set} holds uniformly in all directions $\hat{x}=x/|x|\in\mathbb{S}^{l-1}$ and is known as the Sommerfeld radiation condition.
	
	The system \eqref{set} is well understood, and it is known that there exists a unique solution $u\in H^{1}_{loc}(\mathbb{R}^l)$. Moreover, $u$ is (real) analytic and the asymptotic behaviour of the scattered field is governed by $$u^s(x)=\frac{e^{ik|x|}}{|x|^{(l-1)/2}}u_{\infty}(\hat{x})+O\left(\frac{1}{|x|^{(l+1)/2}}\right),~~|x|\rightarrow\infty,$$
	uniformly in all directions $\hat{x}\in\mathbb{R}^{l-1}$. The analytic function $u_{\infty}$ is defined on the unit sphere $\mathbb{S}^{l-1}$ and called the far-field pattern, which is given by
	\begin{equation}
	u_\infty(\hat{x})=\frac{k^2}{4\pi}\int_{\mathbb{R}^l}e^{-ik\hat{x}\cdot y}q(y)u(y)dy,
	\label{far-pattern}
	\end{equation}
	with $\hat{x}\in\mathbb{S}^{l-1}$ known as the observation angle/direction (see \cite{CK}). In what follows, we write $u_{\infty}(\hat{x},d,k)$ to specify its dependence on the observation direction $\hat{x}$, the incident direction, and the wave number $k$. In the current article, we are concerned with a geometrical inverse problem of reconstructing the shape/support of the scatterer, namely $D$, disregarding its content $q$, by knowledge of the measurement of $u_{\infty}(\hat{x},d,k)$. This problem has been playing an indispensable role in many areas of sciences and technology such as radar and sonar, medical imaging, geophysical exploration, and nondestructive testing (see, e.g., \cite{CK}).
	
	A number of numerical reconstruction methods have been developed for the aforementioned inverse scattering problem in various scenarios, including the linear sampling method, factorization method, MUSIC-type methods, time reversal method, single-shot method and topological-optimization-type method; we refer the readers to \cite{AGJKKY,AGJK,AJGKLS,AGKKPS,AIL,AK,AKbook,CC,CCMbook,CKirs,CC2,KG,LLZ,U} and the references therein for these methods and some other related developments. In particular, a type of so-called interior transmission eigenvalue problem arises in the study of linear sampling and factorization methods. These methods succeed only at wave numbers which are not the transmission eigenvalues. The study of the transmission eigenvalue problem is mathematically interesting and challenging since it is a type of non-elliptic and non self-adjoint problem. Recently, the transmission eigenvalue problems were discovered to be useful in inverse medium scattering problem. In \cite{CCM}, a lower bound on the index of refraction is obtained by knowledge of the smallest real transmission eigenvalue corresponding to the medium. In \cite{CCC}, transmission eigenvalues are used to the nondestructive testing of dielectrics. The aforementioned studies as well as the other existing ones, to name a few \cite{CPS,S,SS}, focus on the intrinsic properties of transmission eigenvalues. In the recent papers \cite{BL,BL1,BLLW}, the intrinsic geometric structure of interior transmission eigenfunctions are studied. Specifically, the vanishing and localizing properties of the interior transmission eigenfunctions at cusps on the support of the refractive index are investigated both theoretically and numerically. Finally, we would like to mention in passing that the study of interior transmission eigenvalue problems was also connected to that of invisibility cloaking in the recent papers \cite{JL,LLLW,LWZ}.
	
	Motivated by the aforementioned works on the intrinsic geometric properties of the interior transmission eigenfunctions, we shall develop a novel scheme for the inverse scattering problem of extracting the geometric structures of an unknown and inaccessible inhomogeneous medium by using the corresponding acoustic far-field measurement. The rationale of the proposed method can be briefly described as follows. Consider the following interior transmission eigenvalue problem associated to $(D, q)$,
	\begin{equation}
	\begin{cases}
	\Delta u+k^2(1+q)u=0~\quad &~\mbox{in}~ D,\\
	\Delta v+k^2v=0~\quad &~\mbox{in} ~D,\\
	u-v\in H^2_0(D),
	\end{cases}
	\label{interior}
	\end{equation}
	where it is noted that if $D$ is a Lipschitz domain
	\begin{align*}
	H_0^2(D) = \left\{v \in H^2(D): v =0, ~ \partial_\nu v = 0 \ {\rm on}\ \partial D \right\},
	\end{align*}
	with $\nu$ signifying the unit normal vector directed into the exterior of $D$.
	
	\begin{defn}
		A number $k \in \mathbb{R}_+, k \neq 0$ for which the transmission problem \eqref{interior} has nontrivial solutions $(u,v) \in L^2(D) \times L^2(D)$ such that $u-v \in H_0^2(D)$ is called an interior transmission eigenvalue associated with $(D; q)$. The nontrivial solutions $(u,v)$ are called the corresponding interior transmission eigenfunctions.
	\end{defn}

	The existence, discreteness and infiniteness of transmission eigenvalues for \eqref{interior} can be found \cite{CKP,CPS,CHad,CGH10} and the references therein. However, there are few results concerning the intrinsic properties of the transmission eigenfunctions. The intrinsic geometric structures of interior transmission eigenfunctions were recently investigated in \cite{BL,BL1,BLLW}. Specifically, it is shown that if there are cusp singularities on $D$ in \eqref{interior}, then the interior transmission eigenfunctions would reveal a certain intrinsic quantitative behaviours near the cusps. Here, by a cusp we mean a point where the tangential vector field of the boundary $\partial D$ is discontinuous. Corner singularities are very typical cusp singularities. To be more precise, in \cite{BL,BLLW}, it is shown that if there is a cusp on the support of $q(x)$, then the transmission eigenfunction vanishes near the cusp if its interior angle is less than $\pi$, whereas the transmission eigenfunction localizes near the cusp if its interior angle is bigger than $\pi$. Furthermore, it is shown that the vanishing and blowup orders are inversely proportional to the interior angle of the cusp: the sharper the angle, the higher the convergence order.

	The proposed reconstruction method in this paper is first to make use of the far-field data $u_\infty(\hat x, d, k )$ to determine the interior transmission eigenvalue as well as the corresponding transmission eigenfunctions. We are aware of some existing study by making use of the so-called inside-outside duality to determine the transmission eigenvalues \cite{LR}. But for our need, we shall also need to determine the corresponding eigenfunctions. The determination of transmission eigenvalue and corresponding eigenfunctions is based on the far-field regularization techniques. To our best knowledge, the study in this aspect is also new to the literature. After the determination of the transmission eigenvalue, we seek the Herglotz wave function in a certain domain which is the approximation of the corresponding transmission eigenfunction. The places where the Herglotz wave function is vanishing or localizing are the locations of those cusp singularities of the support of the medium scatterer. If further a priori information is available on the support of the medium, say, it is a convex polyhedron, then one can actually recover its shape by simply joining the cusp singularities by line. If not much a priori information is given, at least we can recover the locations of the cusp singularities. Extensive numerical experiments show that our reconstruction method is efficient and effective. Our study is first of its type in the literature and opens up a new direction in the study of inverse scattering problems.
	
	The rest of the paper is organized as follows. In Section 2, we review some theoretic results on the properties of transmission eigenvalues and eigenfunctions. In Section 3, we focus on the property of the kernel of far-field operator and the Herglotz approximation to the interior transmission eigenfunctions. In Section 4, we present the recovery scheme in detail.  Section 5 is devoted to numerical experiments to validate the applicability and effectiveness of the proposed method. The paper is concluded in Section 6 with some discussion.

	\section{Preliminaries on transmission eigenvalue problem}
	
	In this section, we first collect some theoretical results for the interior transmission eigenvalue problem \eqref{interior}. Throughout this paper, we assume that the refractive index $n\in L^{\infty}(D)$ is real-valued. Denote by
	\begin{align*}
	n_* = \inf_{x \in D} n(x), \quad n^* = \sup_{x \in D} n(x).
	\end{align*}
	The following theorem in \cite{CGH10} gives the existence of interior transmission eigenvalues.
	\begin{thm}
		Let $n\in L^{\infty}(D)$ satisfy either one of the following assumptions:
		\begin{enumerate}
			\item $1+\alpha \leq n_*\leq n \leq n^* < \infty$,
			\item $0 < n_* \leq n\leq n^* \leq 1-\beta$
		\end{enumerate}
		for some constants $\alpha>0$ and $\beta>0$. Then there exists an infinite set of interior transmission eigenvalues with $+\infty$ as the only accumulation point.
	\end{thm}
	
	Recall from \cite{CGH10} that the lower bound of the first transmission eigenvalue has the following estimation.
	\begin{thm} \label{thm:bounds}
		Let $r$ be the radius of the smallest ball containing $D$. Denote by $k_{1,n_*}$ and $k_{1,n^*}$ the first positive transmission eigenvalue corresponding to the ball of radius $1$ with index of refraction $n_*$ and $n^*$, respectively. Then the first transmission eigenvalue $k_1$ corresponding to $D$ and the given index of refraction $n$ has the following estimations:
		\begin{enumerate}
			\item If $1+\alpha \leq n_*\leq n \leq n^* < \infty$ for some constant $\alpha>0$, then
			\begin{align}\label{esti_geqn}
			k_1 \geq \max \left( \frac{k_{1,n^*}}{r}, \sqrt{\frac{\lambda_1(D)}{n^*}} \right);
			\end{align}
			\item If $0 < n_*\leq n \leq n^* < 1-\beta$ for some constant $\beta>0$, then
			\begin{align}\label{esti_leqn}
			k_1 \geq \max \left( \frac{k_{1,n_*}}{r}, \sqrt{\lambda_1(D)} \right),
			\end{align}
		\end{enumerate}
		where $\lambda_1(D)$ is the first Dirichlet eigenvalue for $-\Delta$ in $D$.
	\end{thm}
	
	The numerical experiments in \cite{BLLW} indicate the existence of complex transmission eigenvalues, but this has not been established theoretically in general. The following theorem shows the non-existence of purely imaginary transmission eigenvalues \cite{CMS10}.
	\begin{thm}
		If $n>1$ or $n<1$ almost everywhere in $D$, then there exists no purely imaginary transmission eigenvalues.
	\end{thm}
	
	Next, we give a more definite description of the vanishing and localizing of transmission eigenfunctions from \cite{BL,BLLW}.
	
	\begin{defn}
		Assume that $k\in\mathbb{R}_+$ is a transmission eigenvalue, then there exist $u,v\in L^2(D)$ such that
		\begin{equation*}
		\begin{cases}
		& \Delta u + k^2 (1+q) u = 0  \quad\, {\rm in}\ D,  \\
		& \Delta v + k^2 v = 0  \quad {\rm in}\ D,  \\
		& u-v\in H^2_0(D),~\|v\|_{L^2(D)}=1.
		\end{cases}
		\end{equation*}
		Let $P\in\partial D$ be a point and $B_r(P)$ be a ball of radius $r\in\mathbb{R}_+$ centered at $P$. Set $D_r(P):=B_r(P)\cap D$. Assume that
		\begin{equation*}
		\|q\|_{L^\infty(D_r(p))}\geq \epsilon_0, \quad\epsilon_0\in\mathbb{R}_+.
		\end{equation*}
		Then we say that vanishing occurs near $P$ if
		$$\lim_{r\rightarrow +0}\frac{1}{\sqrt{|D_r(P)|}}\|v(x)\|_{L^2(D_r(P))}=0,$$
		whereas we say that localizing occurs near $P$ if
		$$\lim_{r\rightarrow +0}\frac{1}{\sqrt{|D_r(P)|}}\|v(x)\|_{L^2(D_r(P))}=+\infty,$$
		where $|D_r(P)|$ signifies the area or volume of the region $D_r(P)$, respectively, in two and three dimensions.
	\end{defn}
	
	The vanishing and localizing property of transmission eigenfunctions near a corner point comes from the corner scattering study in \cite{BL1,BPS,PSV,HSV,EH1,EH2}. If $P$ is the vertex of a corner with an interior angle less than $\pi$, the vanishing of the transmission eigenfunctions has been rigorously verified in \cite{BL}. An important consequence of the study in \cite{BL1,BPS,PSV,EH1} is the fact that the transmission eigenfunction $v$ cannot be analytically extended across a corner point to form an entire solution to the Helmholtz equation, $\Delta v+k^2 v=0$ in $\mathbb{R}^l$. However, we note that due to the interior regularity, the transmission eigenfunction $v$ is always analytic away from the corner point. Hence, heuristically, the failure of the analytic extension may indicate that $v$ either vanishes or blows up when approaching the corner point. Clearly, the failure of the analytical extension should also hold across any irregular point on the support of $q$, and hence in \cite{BLLW}, it is numerically shown that the vanishing or localizing behaviours of the transmission eigenfunctions would occur near any cusps on the support of the underlying. Theoretical proof of such a conjecture is fraught with significant difficulties. Indeed, the proof in \cite{BL} of the vanishing of the transmission eigenfunction in the special case near a corner with an interior angle less than $\pi$ already involves much technical analysis and advanced tools. In \cite{BLLW}, the order of vanishing and localizing convergence rate has been investigate numerically, and show that it is related to the angle of the corner. In the three dimensional case, it turns out that edges also posses the vanishing phenomena.
	
	\section{Property of the kernel of far-field operator and the Herglotz approximation to the interior transmission eigenfunction}
	In this paper, we are concerned with the inverse medium problem of determining the support of unknown domain $D$ with cusp singularities based on the vanishing and localizing property of transmission eigenfunctions. Consider the scattering problem \eqref{set}. We first introduce the far-field operator $F_k$: $L^2(\mathbb{S}^{l-1})\rightarrow L^2(\mathbb{S}^{l-1})$ defined by
	$$F_k(g)(\hat{x}):=\int_{\mathbb{S}^{l-1}}u^{\infty}(\hat{x},d,k)g(d)ds(d), ~\hat{x}\in\mathbb{S}^{l-1}.$$
	The far field operator $F_k$ is injective and has dense range if and only if there does not exist a Dirichlet eigenfunction for $D$ which is a Herglotz wave function (see, for example, \cite{CK}).
	\begin{defn}
		A Herglotz wave function is a function of the form
		\begin{equation}
		\label{Herg}
		H_k(g)(x)=\int_{\mathbb{S}^{l-1}}e^{ikx\cdot d}g(d)ds(d),~~x\in\mathbb{R}^l,
		\end{equation}
		where $g\in L^2(\mathbb{S}^{l-1})$. The function $g$ is called the Herglotz kernel of $H_k(g)$.
	\end{defn}
	We recall from \cite{Weck} the following result concerning the Herglotz approximation .
	\begin{thm}\label{denseness}
		Let $\boldsymbol{W}_k$ denote the space of all Herglotz wave functions of the form \eqref{Herg}. Define, respectively,
		$$W_k(D):=\{u\in C^{\infty}(D):(\Delta+k^2)u=0\},$$
		and
		$$\boldsymbol{W}_k(D):=\{u|_{D}:u\in\boldsymbol{W}_k\}.$$
		Then $\boldsymbol{W}_k(D)$ is dense in $W_k(D)\cap H^1(D)$ with respect to the topology induced by the $H^1$-norm.
	\end{thm}
	
	If $k$ happens to be the interior transmission eigenvalue to the following interior transmission problem
	\begin{equation}
	\begin{cases}
	& \Delta u + k^2 (1+q) u = 0  \quad\, {\rm in}\ D,  \\
	& \Delta v + k^2 v = 0  \quad {\rm in}\ D,  \\
	& u-v\in H^2_0(D).
	\end{cases}
	\end{equation}
	By the denseness property of Herglotz wave functions, see Theorem \ref{denseness}, there exists Herglotz wave function
	$$v_g=\int_{\mathbb{S}^{l-1}}e^{ikx\cdot d}g(d)ds(d),$$
	where the kernel $g\in L^2(\mathbb{S}^{l-1})$, such that for any sufficiently small $\epsilon\in \mathbb{R}^+$, there is
	\begin{equation}\label{Herg_vg}
	\|v_g-v\|_{H^1(D)}\leq \epsilon.
	\end{equation}
	We say that $v_g$ is normalized if $\|g\|_{L^2(\mathbb{S}^{l-1})}=1$. Then $v$ extends to the whole $\mathbb{R}^l$ as the Herglotz wave function $v_g$. It is shown that transmission eigenfunctions cannot be extended analytically to a neighbourhood of a corner (see, for example, \cite{BPS}). Therefore, the transmission eigenfunctions must vanish near the corner point.
	
	In this paper, we aim at finding the support of $q$ when it is supported in a polyhedral domain. If domain $D$ is not of polyhedral shape but possesses cusp point, we aim at locating cusp singularities. The main idea is to find the vanishing and localizing points of the Herglotz wave extension function of transmission eigenfunctions. To do so, we first need to retrieve the transmission eigenvalue from a knowledge of far-field pattern $u_{\infty}$.
	
	Consider the scattering problem \eqref{set}. Let $u^i(x)=e^{ikx\cdot d}$ be the simple incident plane wave, where $d\in\mathbb{R}^{l-1}$ is the incident direction. The far field pattern $u^{\infty}(\hat{x},k,d)$ of the scattered field $u^s$ can be collected at circular boundary which enclose $D$.
	
	By superposition we note that
	\begin{equation}\label{far}
	F_k(g)(\hat{x})=\int_{\mathbb{S}^{N-1}}u^{\infty}(\hat{x},d,k)g(d)ds(d),~~\hat{x}\in\mathbb{S}^{N-1}
	\end{equation}
	is the far-field pattern of the scattered field
	$$v^s(x)=\int_{\mathbb{S}^{N-1}}u^s(x,d)g(d)ds(d)$$
	corresponding to the incident field $v_g$ (see, for example, \cite{CKP}).  For $g\in L^2(\mathbb{S}^{l-1})$, it has the Fourier expansion
	\begin{equation}\label{g}
	g(x)=\sum_{n=0}^{\infty}\sum_{m=-n}^{n}a^m_nY^m_n(\hat{x}),
	\end{equation}
	where $Y^m_n$ denotes a normalized spherical harmonic, and the series is uniformly and absolutely convergent. The coefficients $a^m_n$ are given by
	$$a_n^m=\int_{\mathbb{S}^{l-1}}g(d)\overline{Y^m_n(d)}ds(d).$$
	
	By \eqref{far-pattern} and the reciprocity relation $u_{\infty}(\hat{x},d)=u_{\infty}(-d,-\hat{x})$, see, for example \cite{CK}, we have
	\begin{equation}
	\begin{split}
	F_k(g)(\hat{x})&=\int_{\mathbb{S}^{l-1}}u^{\infty}(-d,-\hat{x})g(d)ds(d)\\
	&=\frac{k^2}{4\pi}\int_{\mathbb{S}^{l-1}}\int_{B}e^{ikd\cdot y}q(y)u(y,-\hat{x})dyg(d)ds(d),
	\end{split}
	\label{Fk}
	\end{equation}
	where $B\in\mathbb{R}^{l}$ is a bounded open ball with radius $R$ enclosing scatter $D$. Inserting \eqref{g} into \eqref{Fk}, then \eqref{Fk} becomes
	\begin{equation}
	\begin{split}
	F_k(g)(\hat{x})&=\frac{k^2}{4\pi}\int_{\mathbb{S}^{l-1}}\int_{B}e^{ikd\cdot y}q(y)u(y,-\hat{x})dy\sum_{n=0}^{\infty}\sum_{m=-n}^{n}a^m_nY^m_n(d)ds(d)\\
	&=\frac{k^2}{4\pi}\int_{B}q(y)u(y,-\hat{x})\sum_{n=0}^{\infty}\sum_{m=-n}^{n}a^m_n\int_{\mathbb{S}^{l-1}}e^{ikd\cdot y}Y^m_n(d)ds(d)dy.
	\end{split}
	\label{Fk1}
	\end{equation}
	By the Funk-Hecke formula (see for example \cite{CK}), \eqref{Fk1} becomes
	\begin{equation}
	F_k(g)(\hat{x})=\frac{k^2}{4\pi}\int_{B}q(y)u(y,-\hat{x})\sum_{n=0}^{\infty}\sum_{m=-n}^{n}a^m_n4\pi i^n j_n(k|y|)Y^m_n(y/|y|)dy,
	\label{Fk2}
	\end{equation}
	where $j_n$ denotes a spherical Bessel function.
	
	With the aid of Stirling's formula $n!=\sqrt{2\pi n}(n/e)^n(1+o(1)),n\rightarrow\infty$, we obtain
	\begin{equation}\label{jnkr}
	j_n(kr)=O\left(\frac{ekr}{2n}\right)^{n},~n\rightarrow \infty
	\end{equation}
	uniformly on $B$. Then we have for any $k\in \mathbb{R}_+$, there holds
	$$F_k(g)(\hat{x})\approx 0$$
	for sufficiently large $N$, by letting $a^m_n=0, n=1,2,\dots,N$, $a^m_n\neq 0, n=N+1,\dots$.
	
	In order to distinguish the far-field operator between two cases that $k$ is a transmission eigenvalue and $k$ is not a transmission eigenvalue, we rigorously prove the following three crucial theorems in detail.
	
	Define
	\begin{equation}\label{gn}
	g_N:=\sum_{n=0}^{N}\sum_{m=-n}^{n}a^m_nY^m_n(\hat{x}),
	\end{equation}
	and
	$$F_{k,N}(g):=\int_{\mathbb{S}^{l-1}}u^{\infty}(\hat{x},d,k)g_N(d)ds(d).$$
	We prove the following theorem.
	\begin{thm}\label{F-Fsmall}	
		If $N$ is sufficiently large, then the following holds
		$$\|F_k(g)-F_{k,N}(g)\|_{L^2(\mathbb{S}^{l-1})}\leq O\left(\frac{ek}{2N}\right)^N,$$
		where $\|g\|_{L^2(\mathbb{S}^{l-1})}=1$.
	\end{thm}
	\begin{proof}
		Let $B\in\mathbb{R}^{l}$ be a bounded open set enclosing scatter $D$. Without loss of generality, we assume that $B$ is a ball centered at $0$ with radius $R$, {\it i.e.}, $B=B(0,R)$.
		
		By \eqref{far-pattern} and \eqref{gn}, $F_{k,N}(g)$ can be expanded as
		\begin{equation}
		F_{k,N}(g)(\hat{x})=\frac{k^2}{4\pi}\int_{B}q(y)u(y,-\hat{x})\sum_{n=0}^{N}\sum_{m=-n}^{n}a^m_n4\pi i^n j_n(k|y|)Y^m_n(y/|y|)dy.
		\label{FkN}
		\end{equation}
		
		From \eqref{Fk2} and \eqref{FkN} we have
		\begin{equation*}
		F_k(g)(\hat{x})-F_{k,N}(g)(\hat{x})=\frac{k^2}{4\pi}\int_{B}q(y)u(y,-\hat{x})\sum_{n=N+1}^{\infty}\sum_{m=-n}^{n}a^m_n4\pi i^n j_n(k|y|)Y^m_n(y/|y|)dy.
		\end{equation*}
		Then
		\begin{equation*}
		\begin{split}
		&\|F_k(g)-F_{k,N}(g)\|^2_{L^2(\mathbb{S}^{l-1})}\\
		&=\frac{k^4}{16\pi^2}\int_{\mathbb{S}^{l-1}}\left|\int_{B}q(y)u(y,-\hat{x})\sum_{n=N+1}^{\infty}\sum_{m=-n}^{n}a^m_n4\pi i^n j_n(k|y|)Y^m_n(y/|y|)dy\right|^2ds(\hat{x})\\
		&\leq \frac{k^4}{16\pi^2}\int_{\mathbb{S}^{l-1}}\sup_{y \in B}\left|q(y)u(y,-\hat{x})\right|\sum_{n=N+1}^{\infty}\sum_{m=-n}^{n}16\pi^2|a^m_n|^2\int_{0}^{R}r^{l-1}|j_n(kr)|^2drds(\hat{x})\\
		&\leq k^4M|\mathbb{S}^{l-1}|\sum_{n=N+1}^{\infty}\sum_{m=-n}^{n}|a^m_n|^2\int_{0}^{R}r^{l-1}|j_n(kr)|^2dr,
		\end{split}
		\end{equation*}
		where $M:=\sup_{y \in B,\hat{x}\in\mathbb{S}^{l-1}}|q(y)u(y,-\hat{x})|$, $|\mathbb{S}^{l-1}|$ is the area of $|\mathbb{S}^{l-1}|$.
		Due to \eqref{jnkr}, we have
		\begin{equation}
		\begin{split}
		&\|F_k(g)-F_{k,N}(g)\|^2_{L^2(\mathbb{S}^{l-1})}\\
		&\leq
		k^4M|\mathbb{S}^{l-1}\|g\|^2_{L^2(\mathbb{S}^{l-1})}\sup_{n>N}\int_{0}^{R}r^{l-1}|j_n(kr)|^2dr\\
		&\leq k^4M|\mathbb{S}^{l-1}|\sup_{n>N}\int_{0}^{R}r^{l-1}|j_n(kr)|^2dr\\
		&= k^4M|\mathbb{S}^{l-1}|\sup_{n>N}O\left(\frac{ek}{2n}\right)^{2n}\int_{0}^{R}r^{2n+l-1}dr\\
		&\leq  k^4M|\mathbb{S}^{l-1}|O\left(\frac{ek}{2N}\right)^{2N}\frac{R^{2N+l}}{2N+l}\\
		&=O\left(\frac{ek}{2N}\right)^{2N}
		\end{split}
		\end{equation}
		for $N$ sufficiently large. The proof is done.
		
	\end{proof}
	
	Let $\mathcal{T}$ denote the set of all the interior transmission eigenvalues. We have the following theorem.
	\begin{thm}\label{kTtheorem}
		If $k\in\mathcal{T}$ with respect to the transmission eigenvalue problem \eqref{interior} ,  then for $\forall\epsilon>0$ and a sufficiently large $N$, there holds
		\begin{equation}\label{kT}
		\|F_{k,N}(g)\|_{L^2(\mathbb{S}^{l-1})}\leq C\epsilon+O\left(\frac{ek}{2N+1}\right)^{N+1},
		\end{equation}
		where $\|g\|_{L^2(\mathbb{S}^{l-1})}=1$.
	\end{thm}
	\begin{proof}
		Let $k\in\mathcal{T}$, and $u,v$ are the corresponding eigenfunctions. For any sufficiently small $\epsilon\in \mathbb{R}^+$, by the denseness property of Herglotz functions, see Theorem \ref{denseness}, there exists Herglotz wave function \eqref{Herg_vg} with $\|g\|_{L^2(\mathbb{S}^{l-1})}=1$, such that for any sufficiently small $\epsilon\in \mathbb{R}^+$, there is
		\begin{equation}
		\|v_g-v\|_{H^1(D)}\leq \epsilon.
		\end{equation}
		
		Denote $v_g$ by $H_k(g)$, and define
		$$H_{k,N}(g)(x)=\int_{\mathbb{S}^{l-1}}e^{ikx\cdot d}g_N(d)ds(d),~~x\in\mathbb{R}^l.$$
		By \eqref{g}, \eqref{gn} and the Funk-Hecke formula, $H_k(g)$ and $H_{k,N}(g)$ can be expanded as
		\begin{equation}\label{Hkg}
		H_k(g)(x)=4\pi\sum_{n=0}^{\infty}\sum_{m=-n}^{n}a^m_ni^nj_n(k|x|)Y^m_n(\hat{x})
		\end{equation}
		and
		$$H_{k,N}(g)(x)=4\pi\sum_{n=0}^{N}\sum_{m=-n}^{n}a^m_ni^nj_n(k|x|)Y^m_n(\hat{x}),$$
		for all $x\in\mathbb{R}^l$. Then
		$$H_k(g)(x)-H_{k,N}(g)(x)=4\pi\sum_{n=N}^{\infty}\sum_{m=-n}^{n}a^m_ni^nj_n(k|x|)Y^m_n(\hat{x}).$$
		
		Similarly to Theorem \ref{F-Fsmall}, let $B\in\mathbb{R}^{l}$ be a bounded open set enclosing scatter $D$. Without loss of generality, we assume that $B$ is a ball centered at $0$ with radius $R$, {\it i.e.}, $B=B(0,R)$. Then
		\begin{equation}
		\begin{split}
		\|H_k(g)-H_{k,N}(g)\|^2_{L^2(B)}&=\sum_{n=N}^{\infty}\sum_{m=-n}^{n}16\pi^2|a^m_n|^2\int_{0}^{R}r^{l-1}|j_n(kr)|^2dr\\
		&\leq 16\pi^2|a^m_n|^2\max_{n\geq N}\int_{0}^{R}r^{l-1}|j_n(kr)|^2dr.
		\end{split}
		\end{equation}
		Since
		$$\int_{R}r^{l-1}|j_n(kr)|^2dr= O\left(\frac{ek}{2N+1}\right)^{2N+2}\frac{R^{2N+l}}{2N+l}\rightarrow 0,$$
		as $N\rightarrow \infty$ for any fixed $R$, then we have
		\begin{equation}\label{H-H}
		\|H_k(g)-H_{k,N}(g)\|_{L^2(B)}\leq O\left(\frac{ek}{2N+1}\right)^{N+1},
		\end{equation}
		for large $N$.
		
		Let $H_{k,N}(g)$ be the incident wave on $D$ and $v_s$ be the corresponding scattered wave. The total wave $v_k=H_{k,N}(g)+v_s$ satisfies
		\begin{equation}
		\begin{cases}
		\Delta v_k+k^2(1+q)v_k=0~\quad &~\mbox{in}~ \mathbb{R}^l,\\
		\lim_{r\rightarrow\infty}r^{(l-1)/2}\left(\frac{\partial v_s}{\partial\nu}-ikv_s\right)=0,&~r=|x|,
		\end{cases}
		\end{equation}
		where $q\equiv 0$ in $\mathbb{R}^l\backslash \bar{D}$. Clearly we can see that $F_{k,N}(g)$ is the far-field pattern of the scattered field $v_s$ corresponding to the incident wave $H_{k,N}(g)$.
		
		Set
		\begin{equation}\label{us}
		u^s=	\begin{cases}
		u-v~\quad &~\mbox{in}~ D,\\
		0~\quad &~\mbox{in}~ \mathbb{R}^l\backslash \bar{D}.
		\end{cases}
		\end{equation}
		Then we have
		\begin{equation}\label{w-v}
		\begin{cases}
		\Delta u^s+k^2(1+q)u^s=k^2qv,\quad &~\mbox{in}~ \mathbb{R}^l,\\
		\lim_{r\rightarrow\infty}r^{(l-1)/2}\left(\frac{\partial u^s}{\partial\nu}-iku^s\right)=0,&~r=|x|.
		\end{cases}
		\end{equation}
		On the other hand,
		\begin{equation}\label{vs}
		\begin{cases}
		\Delta v_s+k^2(1+q)v_s=k^2qH_{k,N}(g)~\quad &~\mbox{in}~ \mathbb{R}^l,\\
		\lim_{r\rightarrow\infty}r^{(l-1)/2}\left(\frac{\partial v_s}{\partial\nu}-ikv_s\right)=0,&~r=|x|.
		\end{cases}
		\end{equation}
		Subtracting \eqref{vs} from \eqref{w-v}, we have
		\begin{equation}\label{posedness}
		\begin{cases}
		\Delta (u^s-v_s)+k^2(1+q)(u^s-v_s)=k^2q(v-H_{k,N}(g))~\quad &~\mbox{in}~ \mathbb{R}^l,\\
		\lim_{r\rightarrow\infty}r^{(l-1)/2}\left(\frac{\partial (u^s-v_s)}{\partial\nu}-ik(u^s-v_s)\right)=0,&~r=|x|.
		\end{cases}
		\end{equation}
		By the well-posedness of the scattering problem \eqref{posedness}, see for example \cite{CK}, we have the following estimate
		$$\|u^s-v_s\|_{H^1(\mathbb{R}^l)}\leq C\|v-H_{k,N}(g)\|_{H^1(\mathbb{R}^l)},$$
		where $C$ is positive constant depending on $k,q$. It follows from \eqref{H-H} that
		$$\|u^s-v_s\|_{H^1(\mathbb{R}^l)}\leq C\epsilon+O\left(\frac{ek}{2N+1}\right)^{N+1}.$$
		Therefore, by \eqref{us} and the well-posedness of the scattering problem, one readily has \eqref{kT}.
	\end{proof}
	
	If $k$ does not belong to the interior transmission eigenvalue class $\mathcal{T}$, we have the following theorem.
	\begin{thm}\label{lower-bound}
		If $k\notin\mathcal{T}$, {\it i.e.}, $dist(k, \mathcal{T})\geq \delta_0$, for some $\delta_0\in\mathbb{R}^+$, then for $\forall g\in L^2(\mathbb{S}^{l-1})$, $\|g\|_{L^2(\mathbb{S}^{l-1})}=1$, there exists $N$ large enough and some constant $C>0$ such that
		$$\|F_{k,N}(g)\|_{L^2(\mathbb{S}^{l-1})}\geq C,$$
		where the constant $C$ depends on $q, k, N, \delta_0$.
	\end{thm}	
	\begin{proof}
		By the expansion \eqref{Hkg}, we have
		\begin{equation*}
		\begin{split}
		\|H_{k,N}(g)\|^2_{L^2(B)}&=\int_{0}^{R}r^{l-1}\int_{\mathbb{S}^{l-1}}|H_{k,N}(g)(rd)|^2ds(d)dr\\
		&=\sum_{n=0}^{N}\sum_{m=-n}^{n}16\pi^2|a^m_n|^2\int_{0}^{R}r^{l-1}|j_n(kr)|^2dr.
		\end{split}
		\end{equation*}
		By Proposition 3.1 in \cite{RS}, and the fact that $j_n(kr)=\sqrt{\pi/(2kr)}j_{n+\frac{1}{2}}(kr)$ we see that
		$$j_n(kr)\geq \frac{(1-\epsilon)e^{n+\frac{1}{2}}k^n}{\sqrt{2}(2n+1)^{n+1}}r^n$$
		for any $0<\epsilon<e^{-1}$, $0<r\leq R$ and $n\geq C(\epsilon,R,k)$, for some constant $C(\epsilon,R,k)$ depending on $\epsilon,R,k$. Then
		$$\int_{0}^{R}r^{l-1}|j_n(kr)|^2dr\geq \frac{(1-\epsilon)^2e^{2n+1}k^{2n}}{2(2n+1)^{2n+2}}\frac{R^{2n+l}}{2n+l}:=\mathcal{C}_n=O\left(\frac{ek}{2n+1}\right)^{2n+2},$$
		and 	
		$$\int_{0}^{R}r^{l-1}|j_n(kr)|^2dr\rightarrow 0,$$
		as $n\rightarrow \infty$ for any fixed $R$. Thus
		\begin{equation}
		\|H_{k,N}(g)\|^2_{L^2(B)}\geq 16\pi^2\|g\|^2_{L^2(\mathbb{S}^{l-1})}\min_{0\leq n\leq N}\int_{0}^{R}r^{l-1}|j_n(kr)|^2dr=16\pi^2\min_{0\leq n\leq N}\mathcal{C}_n.
		\label{Hk}
		\end{equation}
		Note that we consider $H_{k,N}(g)$ for only finite order $N$, then $\mathcal{C}_N$ is a positive number.
		
		We prove that $F_{k,N}(g)$ has a lower bound using the  contradiction argument. Suppose that $F_{k,N}(g)\approx 0$. Since $F_{k,N}(g)$ is the far-field pattern corresponding to the incident wave $H_{k,N}(g)$, then it follows that $H_{k,N}(g_\varepsilon)$ is either zero, or otherwise is a Herglotz wave function approximation to the transmission eigenfunction corresponding to transmission eigenvalue $k$. This contradicts to \eqref{Hk} and $k\notin\mathcal{T}$. The proof is done.	
	\end{proof}
	
	\section{Recovery Scheme}
	Based on our study in the previous section, we shall present a recovery scheme of locating scatters which possesses cusp singularities and reconstructing the shape of penetrable scatter when it is of polyhedral shape. There is no restriction on the size of the scatter in both situations. Our scheme for the reconstruction is based on the intrinsic geometric properties of the transmission eigenfunctions. If the unknown scattering obstacle $D$ has a transmission eigenfunction which is approximated by an entire Herglotz wave function
	$$v_g(x)=\int_{\mathbb{S}^{l-1}}e^{ikx\cdot d}g(d)ds(d)$$
	with proper chosen kernel function $g\in L^2(\mathbb{S}^{l-1})$, where $k$ is the wave number. Then by finding the vanishing and localizing points of the Herglotz wave function we can locate all the cusp singularities of the scatter.
	
	We shall first concentrate on the problem of retrieving the transmission eigenvalues from far field data over a range of test frequency region. Theoretically, we can retrieve transmission eigenvalues from far field data under no a prior information on $D$. In fact, without any a prior assumption on $D$, the range of test frequency could be wide, therefore the computation is a relatively huge work. We recall the fact that the first transmission eigenvalue, the first Dirichlet eigenvalue and the refractive index of scatter have the estimation in Theorem \ref{thm:bounds} in preliminary. Therefore, if we know a priori assumption on the size of domain $D$ and on the upper and lower bounds of refractive index $n$, using the estimation in Theorem \ref{thm:bounds}, we can significantly narrow the searching region of frequency, further reduce the computing cost.

	Given proper searching frequency region $(\alpha,\beta)$, the forward problem \eqref{set} is first solved under the incident wave $u^i=e^{ikx\cdot d}$, where $k\in(\alpha,\beta)$. We collect the far field pattern on the circular boundary which enclose the desired polyhedron for $n$ incident directions and $m$ observation directions, where $n,m$ are positive integers. All the incident directions for each $k$ are uniformly distributed on the unit circle or unit sphere, and all the observation points are uniformly distributed on the closed circle (two-dimensional) or surface (three-dimensional) enclosing the scatterer.
	
	From the discussion in the previous section, if $k_0$ is a transmission eigenvalue, then there should exist integer $N$ sufficiently large and $g\in L^2(\mathbb{S}^{l-1})$ satisfying $\|g\|_{L^2(\mathbb{S}^{l-1})}=1$, such that
	\beq\label{Fg}
	F_{k_0,N}(g)(\hat{x})\approx 0,
	\eeq
	for all incident directions $d$ and observation directions $\hat{x}$. By reciprocity relation, \eqref{Fg} is equivalent to
	$$\int_{\mathbb{S}^{l-1}}u^{\infty}(\hat{x},k_0,d)g_N(\hat{x})ds(\hat{x})\approx 0,$$
	for all incident directions $d$. Then, the work is reduced to be a nonlinear minimization problem. We define the cost functional of the minimization problem as
	\begin{equation}\label{min}
	\mathcal{F}(k,g)=\min_{\|g\|_{L^2(\mathbb{S}^{l-1})}=1}\sum_{d}\Big|F_{k,N}(g)\Big|.
	\end{equation}
	for $N$ sufficiently large. The main issue here is to find proper $k$ and $g$ such that the cost functional is minimized to sufficiently small number. From Theorem \ref{kTtheorem} and \ref{lower-bound}, we see that if the above $k$ can be found, then it must be a transmission eigenvalue. Equipped with the finding transmission eigenvalue $k$ and Herglotz kernel $g$, the Herglotz wave function can be calculated
	\beq\label{Tran-fun}
	v_g(x)=\int_{\mathbb{S}^{l-1}}e^{ik_0x\cdot d}g_N(d)ds(d),~~x\in\mathbb{R}^l,
	\eeq
	where we know it is the extension of corresponding transmission eigenfunction. The locating of cusp singularities then can be achieved by finding the vanishing and localizing points of Herglotz wave function $v_g$ \eqref{Tran-fun}.

	Summarizing the above discussion, we present the scheme of locating cusp singularities scatter $D$ as follows:
	
	\medskip
	
	\noindent {\underline {\bf Recovery Scheme}.}
	
	\medskip
	
	Step 1.~~Fix the searching frequency region $(\alpha,\beta)$ via estimation in Theorem \ref{thm:bounds}, discretize $(\alpha,\beta)$ by step size $h$ as $\alpha,\alpha+h,\alpha+2h,...,\beta$, and set j=0.
	
	Step 2.~~Set $k=\alpha+jh$. Collect the far field pattern on the circular boundary of radius $R$ which enclose the desired polyhedron for $m$ uniformly distributed observation directions $\hat{x}$ and $n$ uniformly distributed incident directions $d$, where the incident directions are uniformly distributed as $d_{\theta}=(\cos\theta,\sin\theta),\theta=0:2\pi/n:2\pi(1-1/n)$, the $m$ observation directions $\hat{x}$ are uniformly distributed on the unit circle or unit sphere.
	
	Step 3.~~Solve the  minimization problem \eqref{min} for $N$ sufficiently large. If the cost functional $\mathcal{F}(k,g)$ can be minimized to zero, then go to the next step; Otherwise, go to step 2.
	
	Step 4.~~Define
	\begin{equation}\label{Herg0}
	v_g(k)=\int_{\mathbb{S}^{l-1}}e^{ikx\cdot d}g(d)ds(d)
	\end{equation}
	to be the Herglotz function (which is the anlytic extension of transmission eigenfunction). Find the vanishing and localizing points of $v_g$. Those points are the cusp singularities of the desired scatter.
	
	If further a priori information is available on the support of the medium, say, it is a convex polyhedron, then we can recover its shape by simply joining the cusp singularities by line. The recovery scheme works in a very general and practical setting. It is no need to give any a priori knowledge of the scatter. There is also no restriction to the size of the scatter.
	
	\section{Numerical experiments and discussions}
	In this section, we present some numerical tests to verify the applicability and effectiveness of the proposed recovery scheme.
	
	we first concentrate on the problem of retrieving the transmission eigenvalues and eigenfunctions from far field data over a range of frequencies. We collect far field data by solving the forward equation \eqref{set} under incident wave $u^i=e^{ikx\cdot d}$, for each $k\in (\alpha,\beta)$ with step $h=0.01$ and each observation direction $d$. The forward problem is solved by using the quadratic finite element discretization on a truncated circular (two-dimensional) or spherical (three-dimensional) domain enclosed by a perfectly matched layer (PML). The forward equation is solved on a sequence of successively refined meshes till the relative error of two successive finite element solutions between the two adjacent meshes is below $0.1\%$. The scattered data are transformed into the far-field data by employing the Kirchhoff integral formula on a closed circle (two-dimensional) or surface (three-dimensional) enclosing the scatterer. We collect the far field pattern on the circular boundary which enclose the desired polyhedral scatter for incident directions $d$ and observation directions $\hat{x}$ (precisely, $128$ directions for each $d$ uniformly distributed on the unit circle in 2D or unit sphere in 3D, $64$ observation points uniformly distributed on the inner part of PML) and a range of wave numbers $k\in(\alpha,\beta)$. Solve the minimization problem \eqref{min} and find the desired $k$ and $g$. The minimizer of the minimization problem \eqref{min} is obtained by employing a derivative-free trust region method via a local quadratic surrogate model-based search algorithm. Finally, the vanishing and localizing points of Herglotz wave equation \eqref{Herg0} give us the location of cusp singularities.

	
	The following numerical experiments consist two parts, one is to locate penetrable scatter of general shape with cusp points, the other is to reconstruct the support of polyhedral type shape.
	
	\subsection{Locating of general shaped scatters with cusps}
	
	\subsubsection{Example: rain drop shape}
	The true scatter $D$ is of rain drop shape with boundary $\partial D$ illustrated in Figure \ref{fig:rain},
	\begin{figure}[t]
		\centering
		\includegraphics[width=0.45\textwidth]{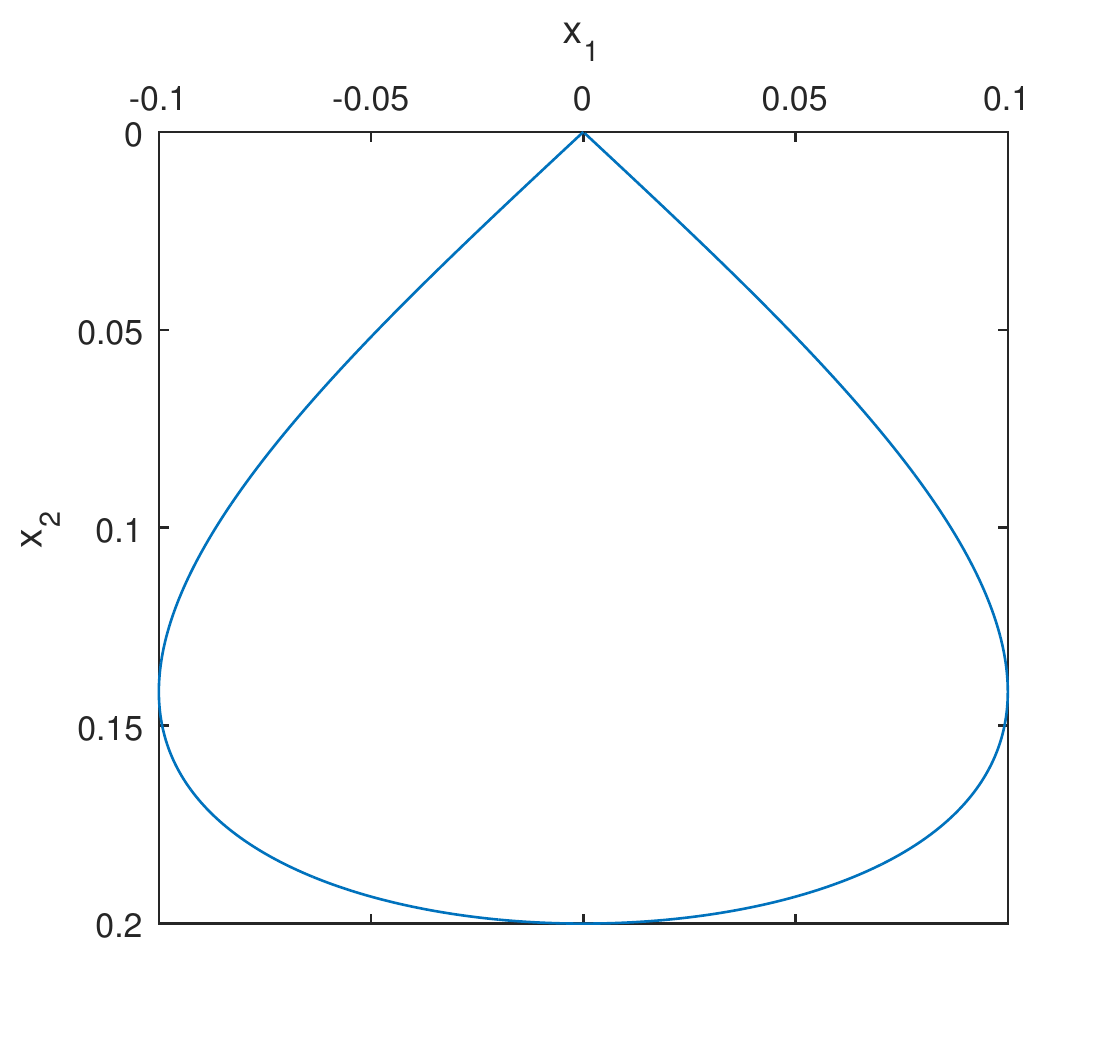}
		\caption{Rain shaped domain}
		\label{fig:rain}
	\end{figure}
	and described by $270^{\circ}$ rotation of the parametric representation
	$$x(t)=(\frac{1}{5}\sin\frac{t}{2},-\frac{1}{10}\sin t),~0\leq t\leq 2\pi.$$
	The location of the corner is $(0,0)$. The refractive index of $D$ is constant $n=4$ which is not known a prior.  From the estimation \eqref{esti_geqn}, we define the searching frequency region to be $(30,35)$ and the searching step is $h=0.01$. We find that when $k=32.14$, the minimization problem \eqref{min} achieve its minimum. We retrieve the corresponding density function $g$. The corresponding Herglotz function is shown in Figure \ref{fig:drop_recon} (a) and it has zero values at $(-0.08,0)$ shown in Figure \ref{fig:drop_recon} (b).
	\begin{figure}[t]
		\centering
		\begin{subfigure}[b]{0.48\textwidth}
			\includegraphics[width=\textwidth]{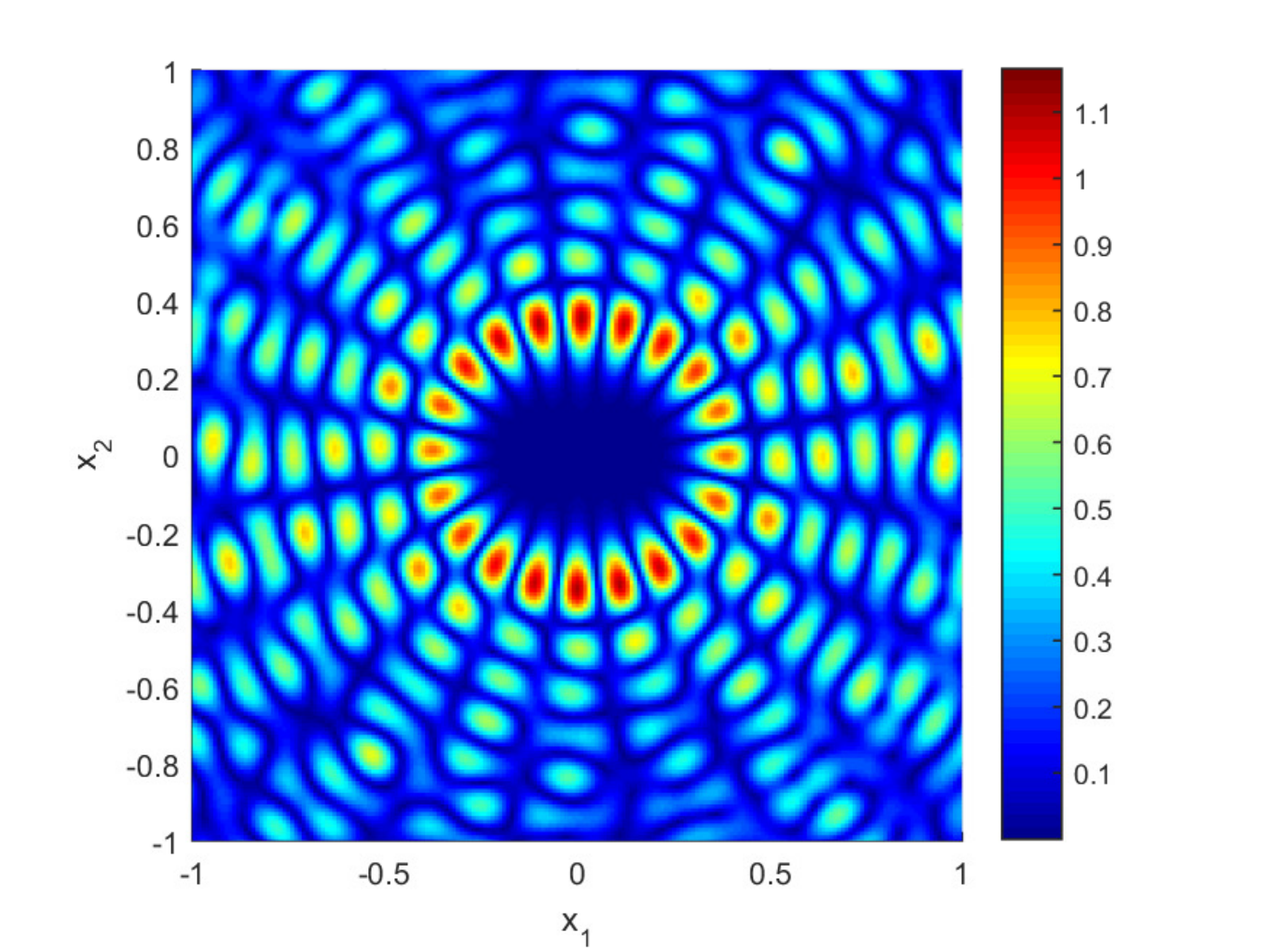}
			\caption{}
		\end{subfigure}
		\hfill
		\begin{subfigure}[b]{0.48\textwidth}
			\centering
			\includegraphics[width=\textwidth]{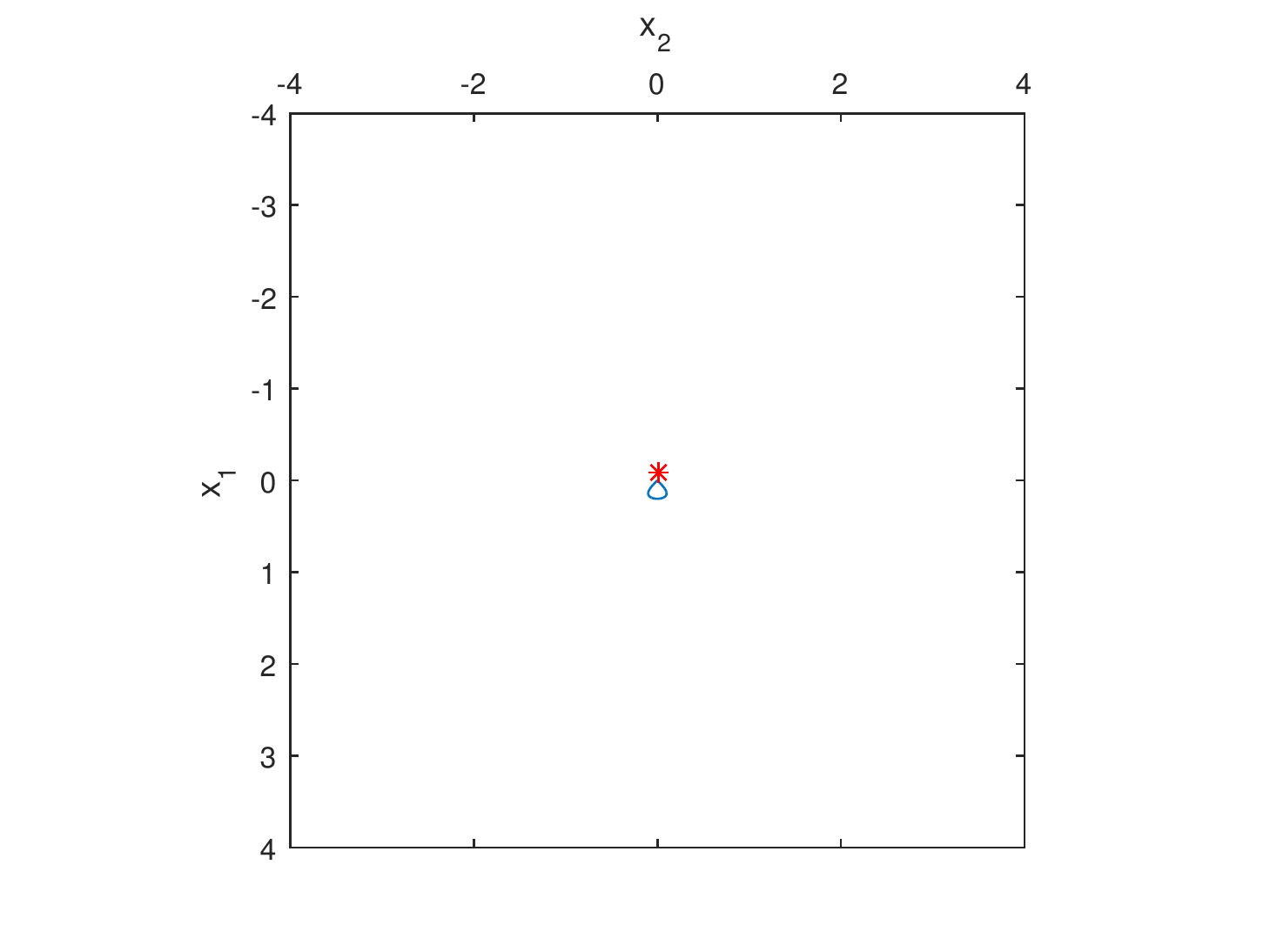}
			\caption{}
		\end{subfigure}
		\caption{(a) The Herglotz wave function approximation of interior transmission eigenfunction. (b) The locations of zero values of Herglotz wave function.}
		\label{fig:drop_recon}
	\end{figure}
	The result shows that our reconstruction location is very close to the real corner location $(0,0)$.
	
	
	\subsubsection{Example: rain drop of regular size}
	The true scatter $D$ is of rain drop shape with boundary $\partial D$ illustrated in Figure \ref{fig:rain_shape},
	\begin{figure}[t]
		\centering
		\includegraphics[width=0.45\textwidth]{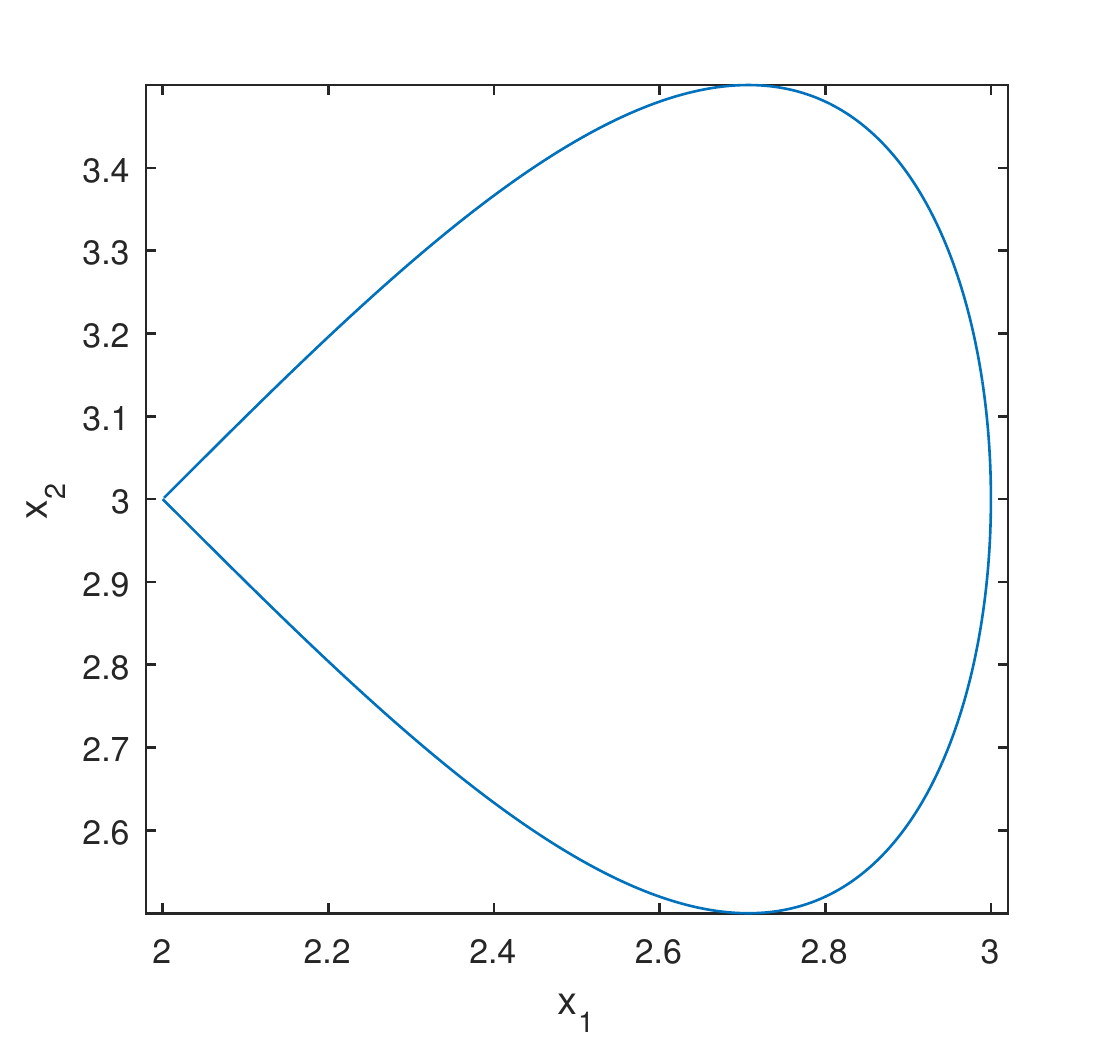}
		\caption{Rain shaped domain}
		\label{fig:rain_shape}
	\end{figure}
	and described by the parametric representation
	$$x(t)=(\sin\frac{t}{2},-\frac{1}{2}\sin t),~0\leq t\leq 2\pi.$$
	The location of the corner is $(2,3)$. The refractive index of $D$ is constant $n=4$ which is not known in a prior.  From the estimation \eqref{esti_geqn}, we define the searching frequency region to be $(5,7)$ and the searching step is $h=0.01$. We find that when $k=6.43$, the minimization problem \eqref{min} achieve its minimum. We retrieve the corresponding density function $g$. The corresponding Herglotz function is shown in Figure \ref{fig:drop_regular_recon} (a) and it has zero values at $(2.32,3.04)$ shown in Figure \ref{fig:drop_regular_recon} (b).
	\begin{figure}[t]
		\centering
		\begin{subfigure}[b]{0.48\textwidth}
			\includegraphics[width=\textwidth]{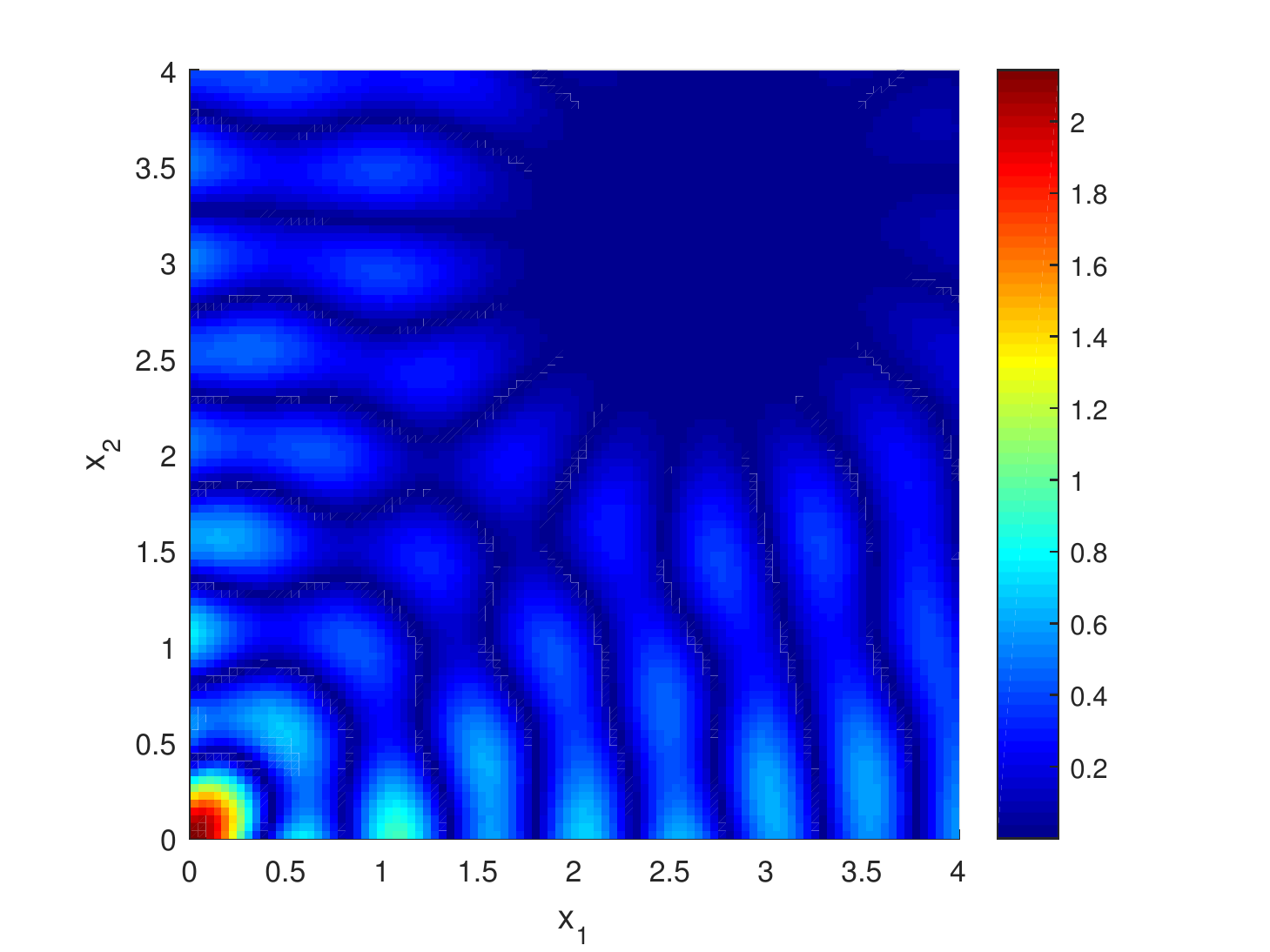}
			\caption{}
		\end{subfigure}
		\hfill
		\begin{subfigure}[b]{0.48\textwidth}
			\centering
			\includegraphics[width=\textwidth]{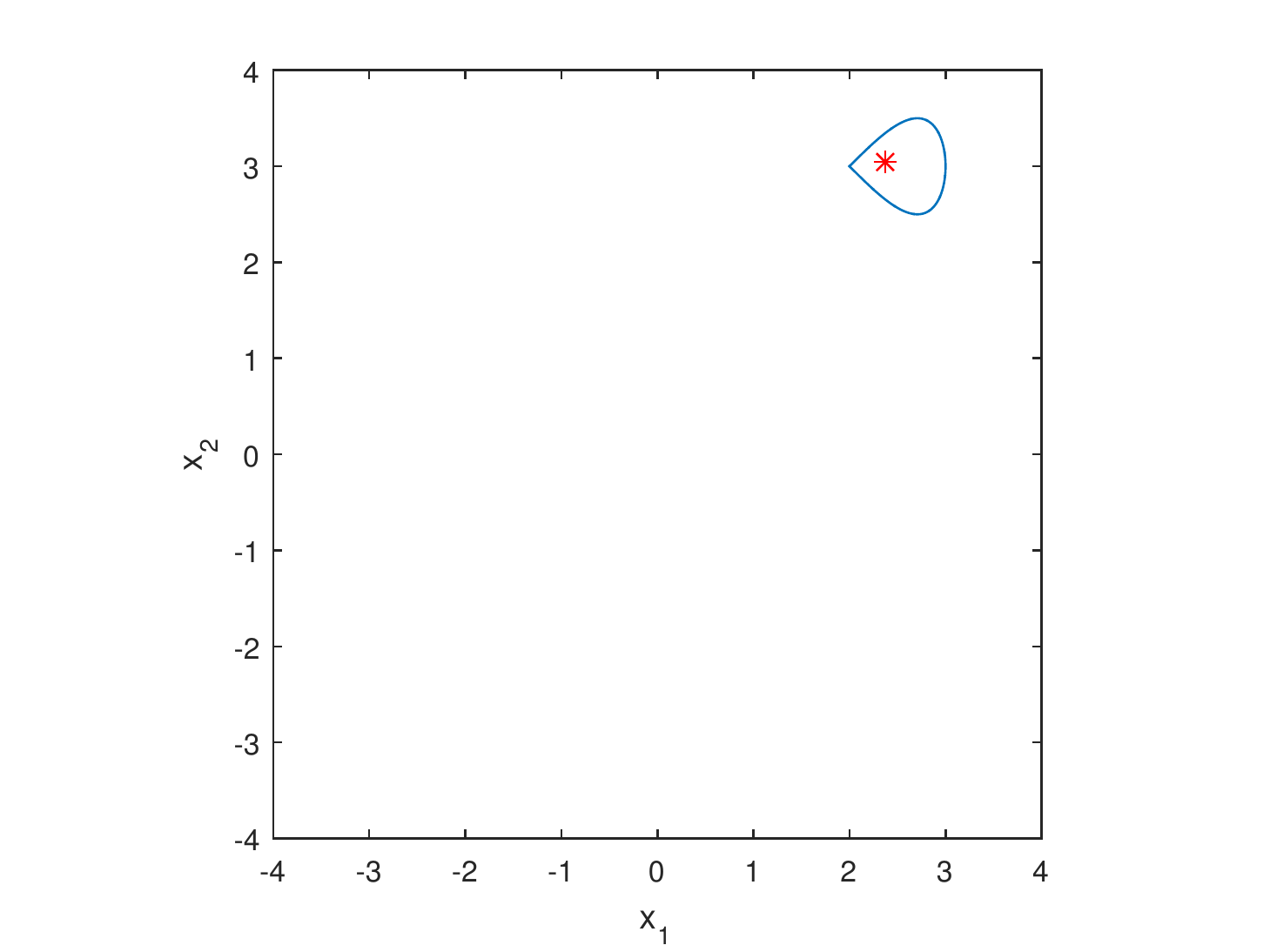}
			\caption{}
		\end{subfigure}
		\caption{(a) The Herglotz wave function approximation of interior transmission eigenfunction. (b) The locations of zero values of Herglotz wave function.}
		\label{fig:drop_regular_recon}
	\end{figure}
	
	The result shows that our reconstruction location is very close to the real corner location $(2,3)$ even when the scatter is not small.

	\subsubsection{Example: Heart shape}
	Let $D$ be a heart shape domain with boundary $\partial D$ illustrated in Figure \ref{fig:heart},
	\begin{figure}[t]
		\centering
		\includegraphics[width=0.45\textwidth]{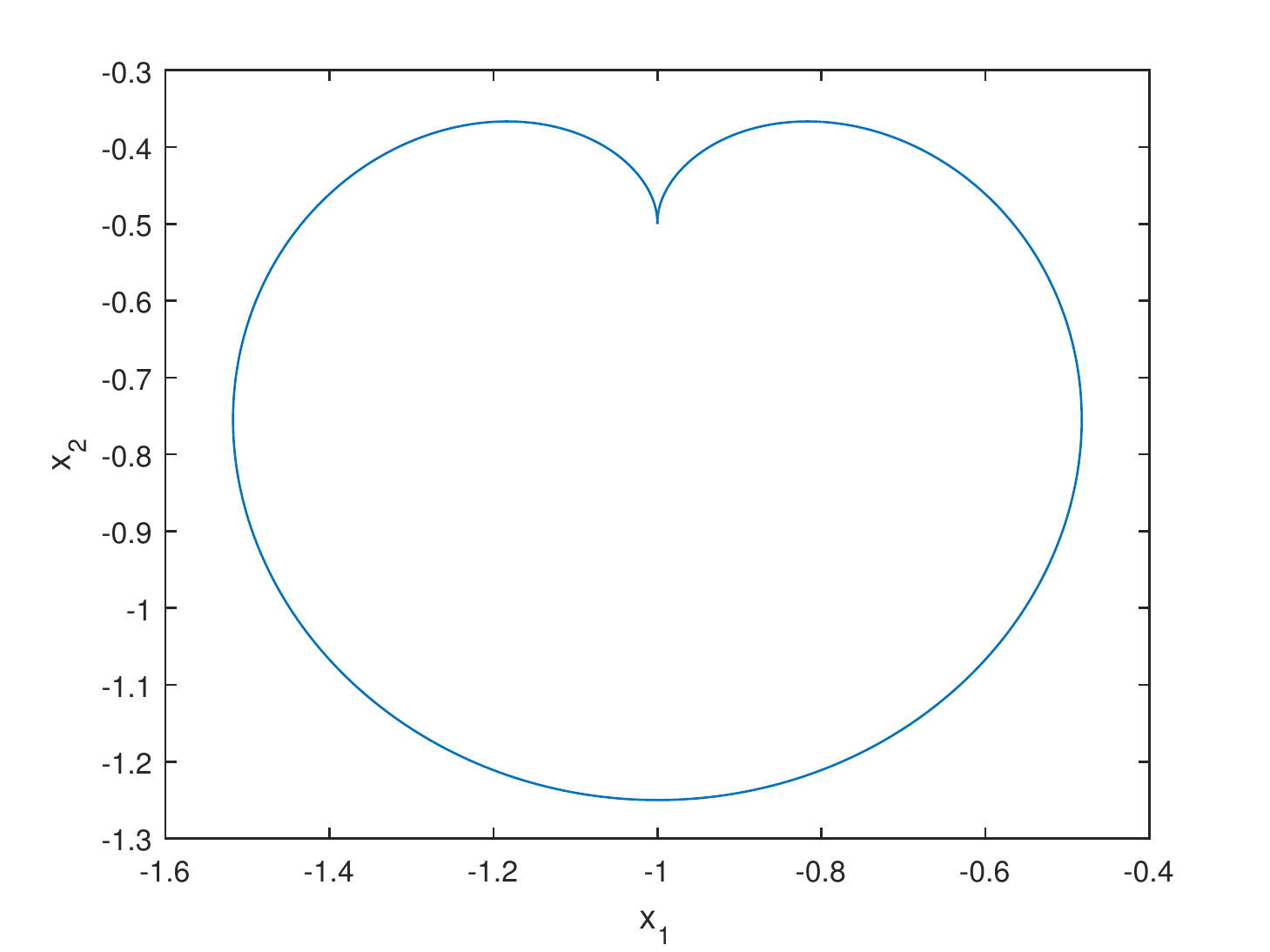}
		\caption{heart shaped domain}
		\label{fig:heart}
	\end{figure}
	and described by the parametric representation
	$$x(t)=((1-\cos t)(1.5\sin t-0.5\sin 2t)/4-1,(1-\cos t)(\cos t-0.5\cos 2t)/4-0.5),$$
	where $0\leq t\leq 2\pi$. The location of the corner is $(-1,-0.5)$. The refractive index of $D$ is constant $n=16$ which is not known in a prior.  From the prior estimate \eqref{esti_geqn}, we define the searching frequency region to be $(2,3)$ and the searching step is $h=0.01$. We find that when $k=2.14$, the minimization problem \eqref{min} achieve its minimum. We retrieve the corresponding density function $g$. The corresponding Herglotz wave function is shown in Figure \ref{fig:heart_recon} (a).
	
	From Figure \ref{fig:heart_recon} (b) we can find the localizing point of Herglotz wave function at $(-1,-0.52)$. But this localizing point is a local maximum. In order to locate it accurately, we need to narrow the region of Herglotz wave function around the cusp singularity.
	\begin{figure}[t]
		\centering
		\begin{subfigure}[b]{0.48\textwidth}
			\includegraphics[width=\textwidth]{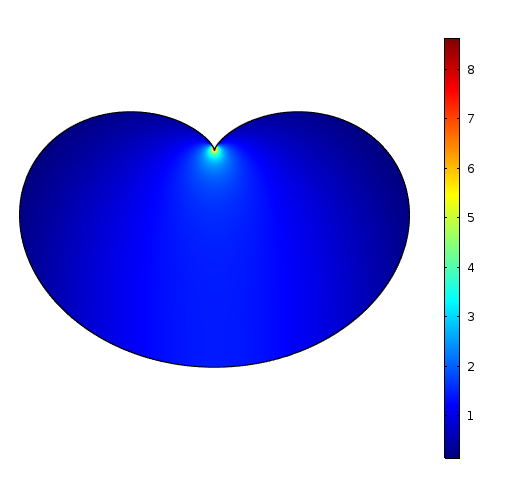}
			\caption{}
		\end{subfigure}
		\hfill
		\begin{subfigure}[b]{0.48\textwidth}
			\centering
			\includegraphics[width=\textwidth]{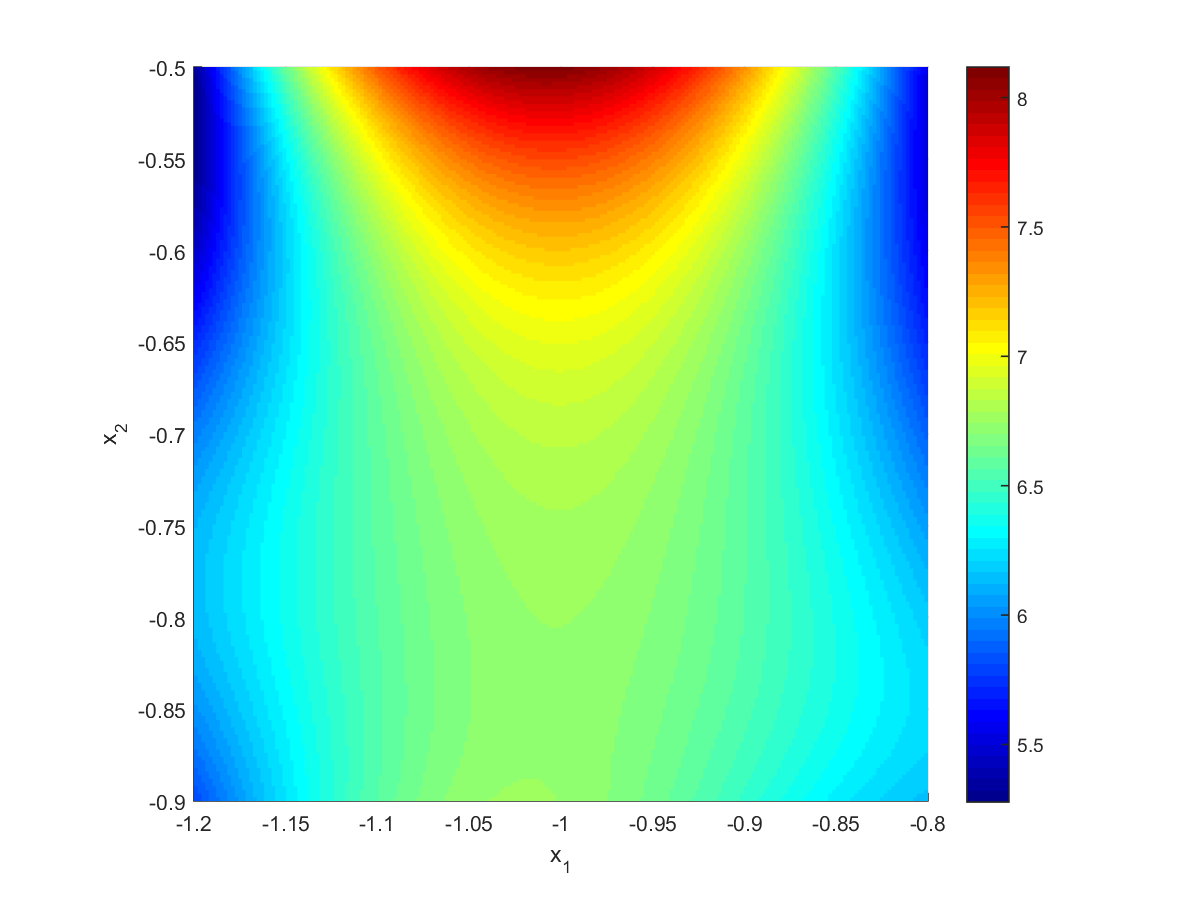}
			\caption{}
		\end{subfigure}
		\caption{(a) Transmission eigenfunction. (b) The Herglotz wave function approximation of interior transmission eigenfunction.}
		\label{fig:heart_recon}
	\end{figure}

	\subsection{Recovery of polyhedrons}
	
	\subsubsection{Example: Square}
	The true scatter $D$ is a square centered at origin of edge length $2$. The refractive index of the square is $n=16$. These are not known in a priori. In this example, we aim to reconstruct the shape of square by locating the four corners of it. In fact, the first three real transmission eigenvalues are computed to be $0.9398, 1.2221, 1.2221$ with corresponding eigenfunctions shown in Figure \ref{fig:square}.
	\begin{figure}[t]
		\centering
		\begin{subfigure}[b]{0.3\textwidth}
			\includegraphics[width=\textwidth]{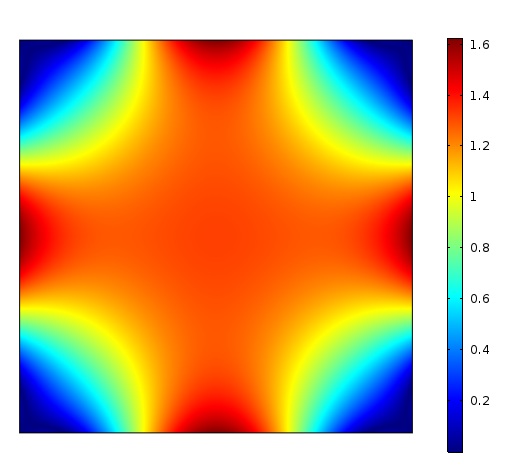}
			\caption{}
		\end{subfigure}
		\hfill
		\begin{subfigure}[b]{0.3\textwidth}
			\includegraphics[width=\textwidth]{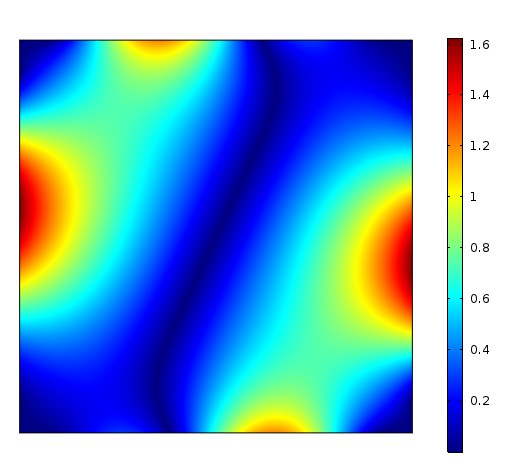}
			\caption{}
		\end{subfigure}
		\hfill
		\begin{subfigure}[b]{0.3\textwidth}
			\includegraphics[width=\textwidth]{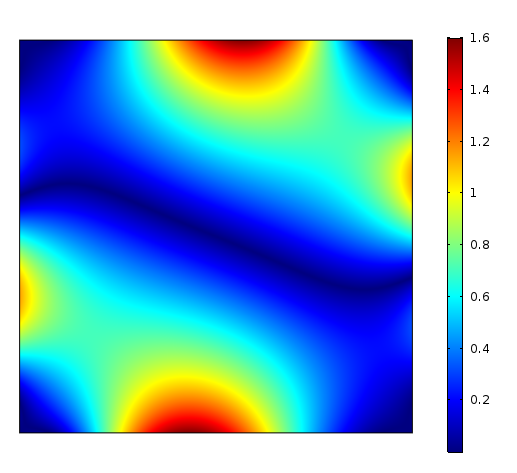}
			\caption{}
		\end{subfigure}
		\caption{The magnitude of transmission eigenfunctions for the square domain with potential $n=16$ of different eigenvalues. (a) $|v|:k_1$; (b) $|v|:k_2$; (c) $|v|:k_2$.}
		\label{fig:square}
	\end{figure}
	
	The transmission eigenvalues are not known in advance. We suppose that $D$ is a convex polygon and the refractive index of $D$ is around $16$. By the estimation \eqref{esti_geqn}, we search the frequency region $(0.5,2)$. It turns out that when $k=0.94$, the minimization problem \eqref{min} achieve its minimum value zero. We also retrieve the corresponding density function $g$.
	
	The approximated Herglotz wave function for transmission eigenfunction is shown in Figure \ref{fig:square_k1} (a). The Herglotz wave function is the extension of the transmission eigenfunction shown in Figure \ref{fig:square} (a). We only display the Herglotz wave function in region $[-1.1,1.1]\times [-1.1,1.1]$ for convenience of comparison between \ref{fig:square} (a) and \ref{fig:square_k1} (a). We find the zero values of the Herglotz wave function in the region $[-4,4]\times [-4,4]$ shown in Figure \ref{fig:square_k1} (b).
	\begin{figure}[t]
		\centering
		\begin{subfigure}[b]{0.48\textwidth}
			\includegraphics[width=\textwidth]{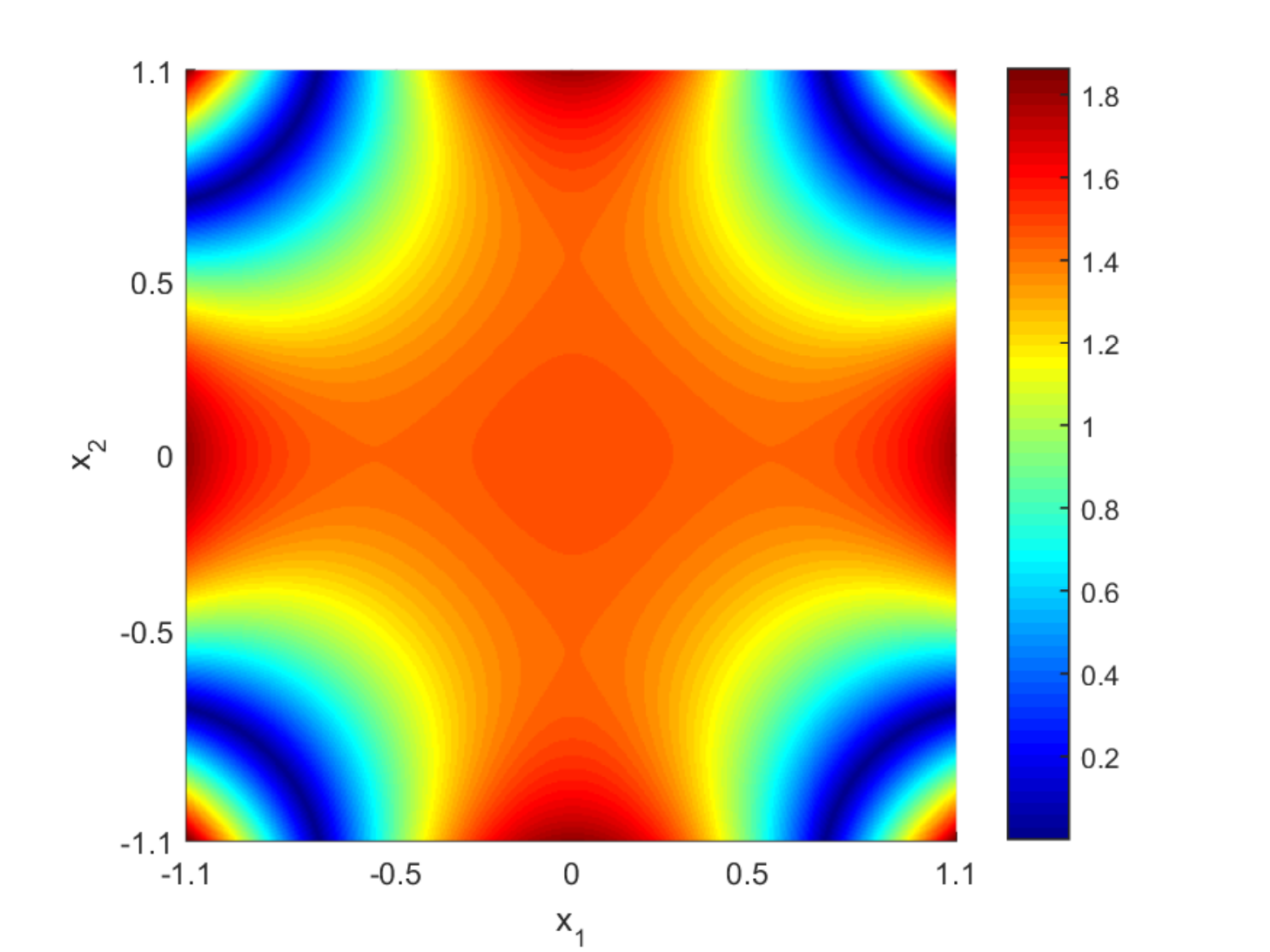}
			\caption{}
		\end{subfigure}
		\hfill
		\begin{subfigure}[b]{0.48\textwidth}
			\includegraphics[width=\textwidth]{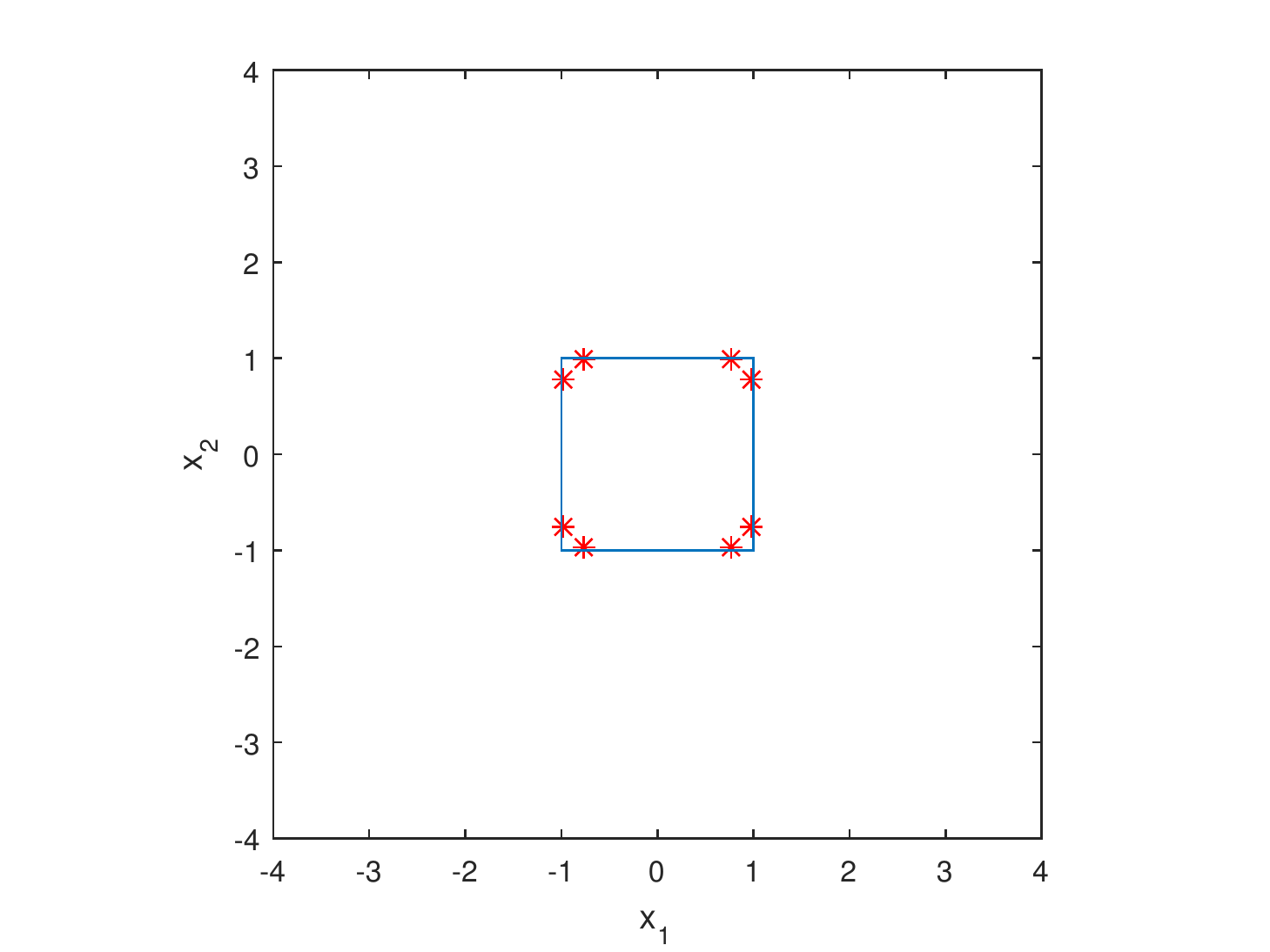}
			\caption{}
		\end{subfigure}
		\caption{(a) The Herglotz wave function approximation of interior transmission eigenfunction. (b) The locations of zero values of Herglotz function.}
		\label{fig:square_k1}
	\end{figure}
	
	In Figure \ref{fig:square_k1} (b), there are two cluster points around each corner, we take average of each cluster points as the reconstruction corner point. The reconstruction of unknown polygon is shown in Figure \ref{fig:square_reco}. The blue dotted square is the real scatter, the red dotted one is our reconstructed result.
	\begin{figure}[t]
		\centering
		\includegraphics[width=0.5\textwidth]{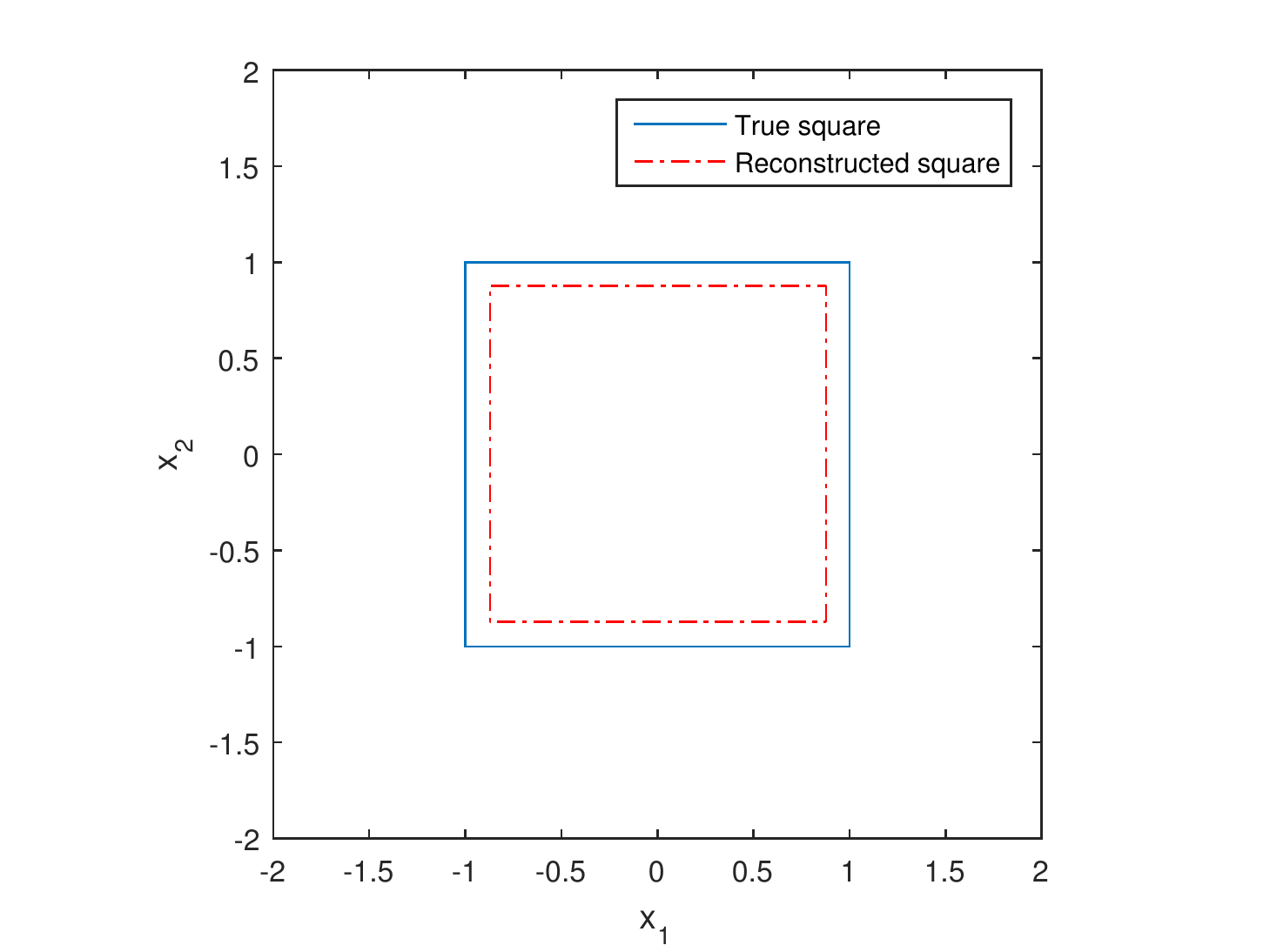}
		\caption{Reconstruction of square}
		\label{fig:square_reco}
	\end{figure}
	
	Note that the wavelength of the probing incident wave is $2\pi/k=6.68$, the length of the square is $2$. The size of the scatter is less than half of the wavelength. Surprisingly, our method turns out to have the super-resolution effect.
	
	\subsubsection{Example: Hexagon}
	The true scatter $D$ is a hexagon centered at origin of edge length $2$. The refractive index of the square is $n=25$. These are not known in a priori. In this example, we aim to locate the six corners of the hexagon by locating the corners of it. In fact, the first three real transmission eigenvalues are computed to be $0.4392, 0.5809, 0.5809$ with corresponding eigenfunctions shown in Figure \ref{fig:hexagon}.
	\begin{figure}[t]
		\centering
		\begin{subfigure}[b]{0.32\textwidth}
			\includegraphics[width=\textwidth]{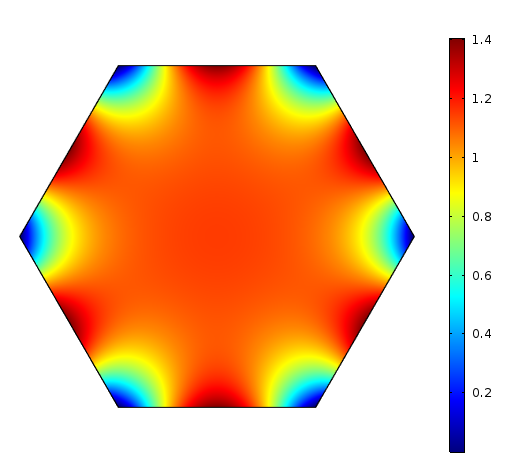}
			\caption{}
		\end{subfigure}
		\hfill
		\begin{subfigure}[b]{0.32\textwidth}
			\includegraphics[width=\textwidth]{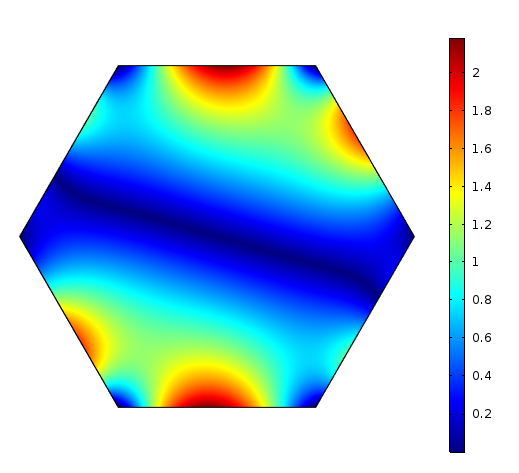}
			\caption{}
		\end{subfigure}
		\hfill
		\begin{subfigure}[b]{0.32\textwidth}
			\includegraphics[width=\textwidth]{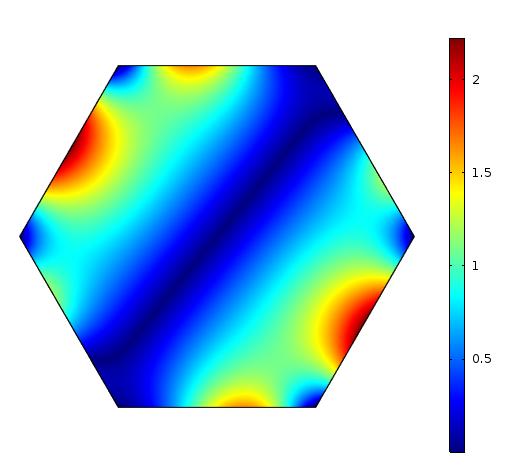}
			\caption{}
		\end{subfigure}
		\caption{The magnitude of transmission eigenfunctions for the hexagon domain with potential $n=25$ of different eigenvalues. (a) $|v|:k_1$; (b) $|v|:k_2$; (c) $|v|:k_2$.}
		\label{fig:hexagon}
	\end{figure}
	
	The transmission eigenvalues are not known in advance. In this example, we aim to reconstruct the support of $D$ under the assumption that $D$ is a convex polygon and the refractive index of $D$ is around $25$. By the estimate \eqref{esti_geqn}, we search the frequency region $(0,1)$. It turns out that when $k=0.44$, the minimization problem \eqref{min} achieve its minimum value $0$. We retrieve the corresponding density function $g$.
	
	The corresponding Herglotz wave function is shown in Figure \ref{fig:hexa_k1} (a) (extension of transmission eigenfunction shown in Figure \ref{fig:hexagon} (a)). For convenience of comparison between Figure \ref{fig:hexagon} (a) and Figure \ref{fig:hexa_k1} (a)), we only display the Herglotz wave function in region $[-2.1,2.1]\times [-2.1,2.1]$. We find the zero values of the Herglotz wave function in the region $[-4,4]\times [-4,4]$ shown in Figure \ref{fig:hexa_k1} (b).
	\begin{figure}[t]
		\centering
		\begin{subfigure}[b]{0.45\textwidth}
			\includegraphics[width=\textwidth]{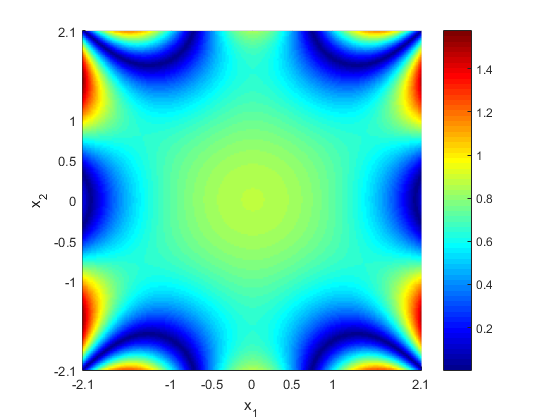}
			\caption{}
		\end{subfigure}
		\hfill
		\begin{subfigure}[b]{0.45\textwidth}
			\includegraphics[width=\textwidth]{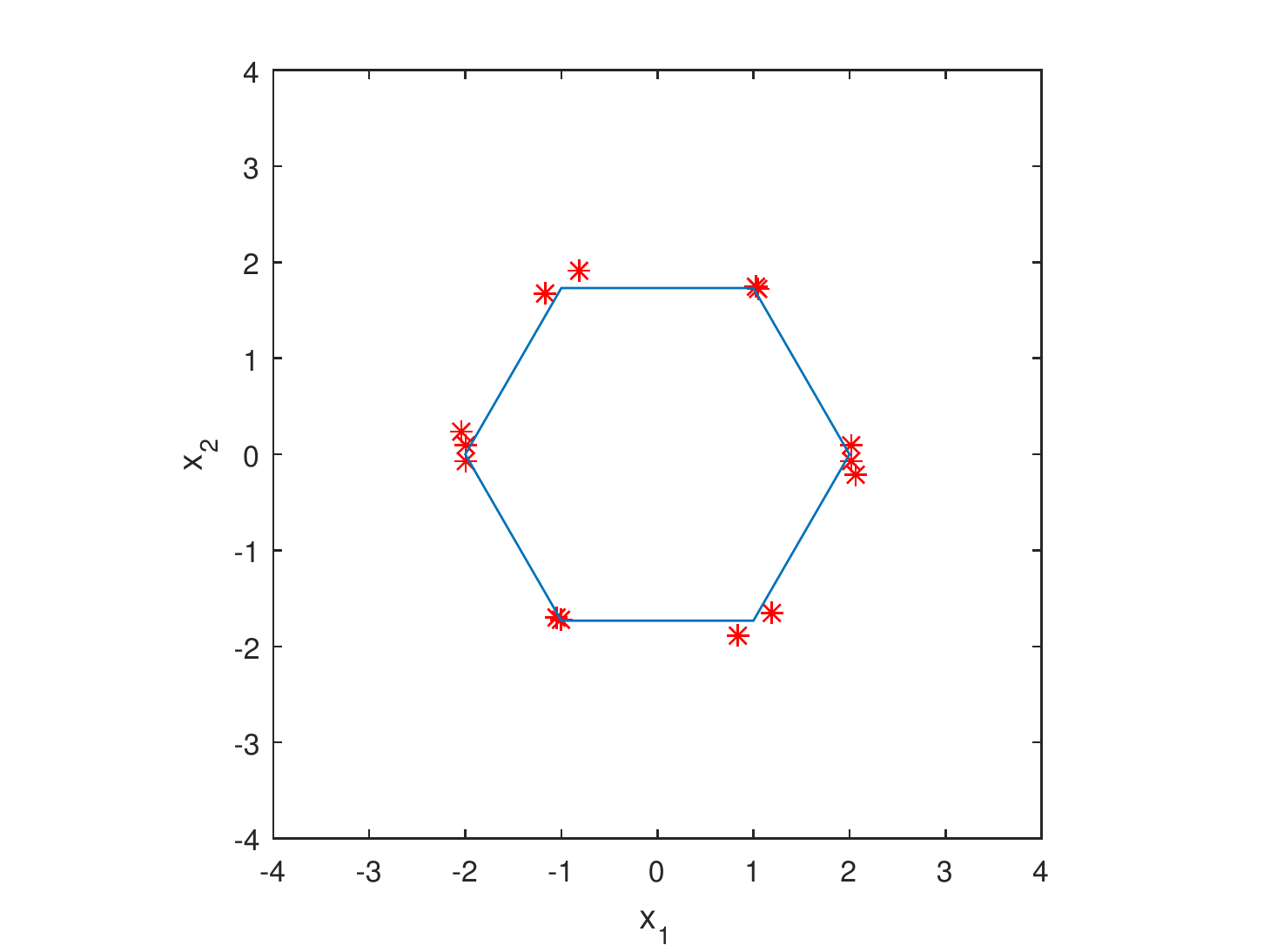}
			\caption{}
		\end{subfigure}
		\caption{(a) The Herglotz wave function approximation of interior transmission eigenfunction. (b) The locations of zero values of Herglotz function.}
		\label{fig:hexa_k1}
	\end{figure}
	We calculate the average of each cluster points. The reconstruction of unknown polygon is shown in Figure \ref{fig:Hexagon_reco}. The blue dotted hexagon is the real scatter, the red dotted one is our reconstructed result. Figure \ref{fig:Hexagon_reco} shows that our method give accurate reconstruction result. The wavelength of the probing incident wave is $2\pi/k=14.27$, the size of the hexagon is $4$. The size of the scatter is less than half of the wavelength. This example also turns out to have the super-resolution effect.
	\begin{figure}[t]
		\centering
		\includegraphics[width=0.5\textwidth]{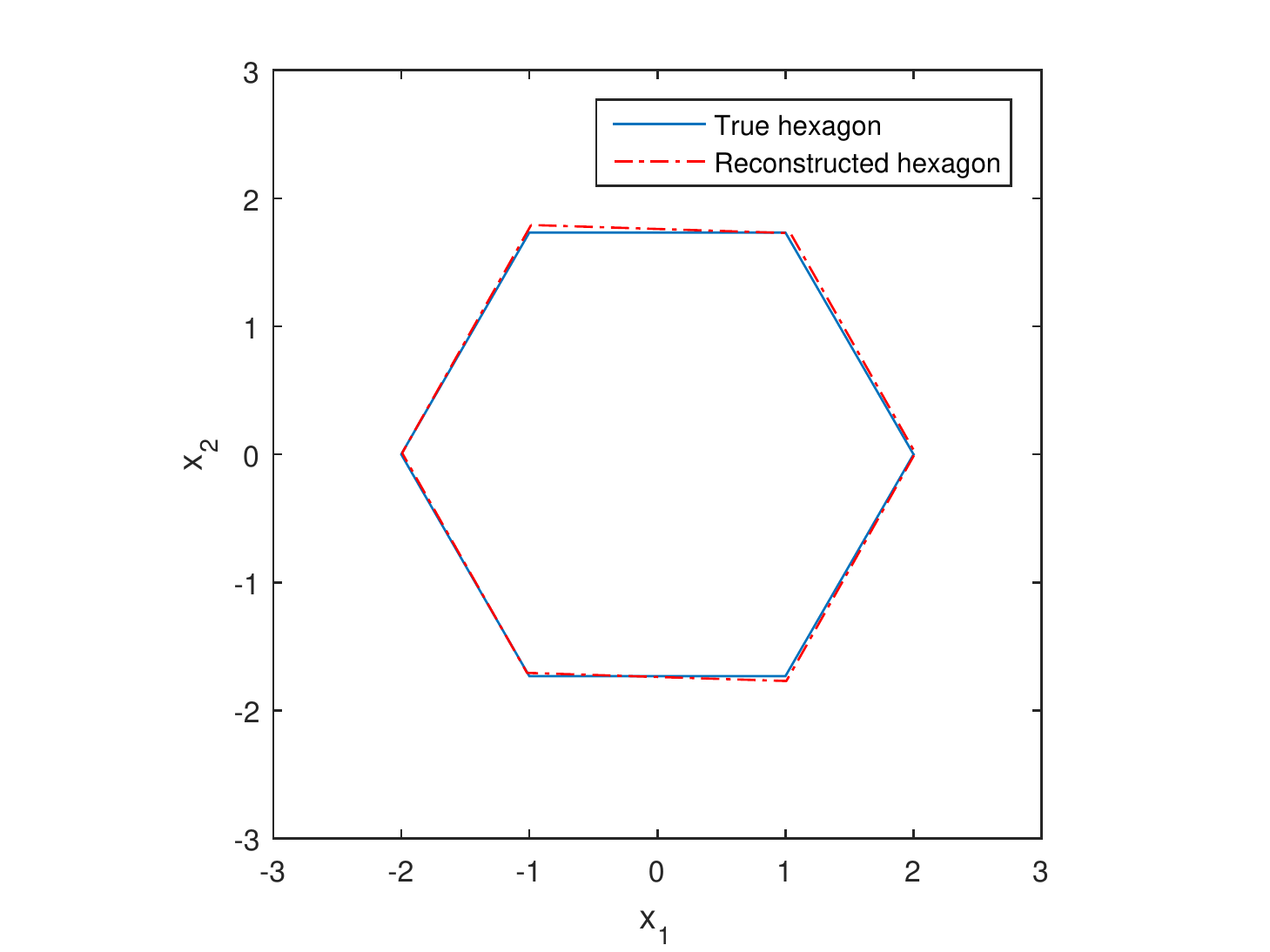}
		\caption{Reconstruction of hexagon}
		\label{fig:Hexagon_reco}
	\end{figure}

	\section{Conclusion}
	In this paper, we develop a novel inverse scattering scheme of extracting the geometric structures of an unknown/inaccessible inhomogeneous medium using the intrinsic geometric properties of the so-called interior transmission eigenfunctions if the scatter posseses some cusp singularities. The proposed reconstruction method is first to make use of the far-field data $u_\infty(\hat x, d, k )$ to determine the interior transmission eigenvalue as well as the corresponding transmission eigenfunctions. The determination of transmission eigenvalue and corresponding eigenfunctions is based on the far-field regularization techniques. To our best knowledge, the study in this aspect is new to the literature. After the determination of the transmission eigenvalue, we seek the Herglotz wave function in a certain domain which is the approximation of the corresponding transmission eigenfunction. The places where the Herglotz wave function is vanishing or localizing are the locations of those cusp singularities of the support of the medium scatterer. If further a priori information is available on the support of the medium, say, it is a convex polyhedron, we can actually recover its shape by simply joining the cusp singularities by line. Surprisingly, our method even works when the size of scatter is greater than the probing wavelength, which encourages us to extend our study to the applications of super-resolution. Our study is first of its kind in the literature and opens up a new direction in the study of inverse scattering problems. In the numerical experiments, we mainly focus on the convex domain in two dimension, we will leave the recovery of non-convex domains as well as the three dimensional numerical experiments in a forthcoming paper.

\section*{acknowledgement}
	The work of Jingzhi Li was supported by the NSF of China (No.~11571161) and the Shenzhen Sci-Tech (No.~JCYJ20160530184212170). The work of Hongyu Liu was supported by the startup fund and the FRG grants from Hong Kong Baptist University, and the Hong Kong RGC grant (No.~12302415).


\begin{thebibliography}{99}
		\bibitem{AGJKKY} {\sc H. Ammari, P. Garapon, F. Jouve, H. Kang, M. Lim, and S. Yu}, {\em A new optimal control approach for the reconstruction of extended inclusions}, SIAM J. Control Optim., 51 (2013), pp. 1372–1394.
		
		\bibitem{AGJK} {\sc H. Ammari, J. Garnier, V. Jugnon, and H. Kang}, {\em Stability and resolution analysis for a topological derivative based imaging functional}, SIAM J. Control Optim., 50 (2012), pp.
		48–76.
		
		\bibitem{AJGKLS} {H. Ammari, J. Garnier, H. Kang, M. Lim, and K. Sølna}, {\em Multistatic imaging of extended targets}, SIAM J. Imaging Sci., 5 (2012), pp. 564–600.
		
		\bibitem{AGKKPS} {\sc H. Ammari, J. Garnier, H. Kang, W. K. Park, and K. Sølna}, {\em Imaging schemes for perfectly conducting cracks}, SIAM J. Appl. Math., 71 (2011), pp. 68–91.
		
		\bibitem{AIL} {\sc H. Ammari, E. Iakovleva, and D. Lesselier}, {\em A MUSIC algorithm for locating small inclusions buried in a half-space from the scattering amplitude at a fixed frequency}, Multiscale
		Model. Simul., 3 (2005), pp. 597–628.
		
		\bibitem{AK} {\sc H. Ammari and H. Kang}, {\em Reconstruction of Small Inhomogeneities from Boundary Measurements}, Lecture Notes in Math. 1846, Springer-Verlag, Berlin, 2004.
		
		\bibitem{AKbook} {\sc H. Ammari and H. Kang}, {\em Polarization and Moment Tensors. With Applications to Inverse Problems and Effective Medium Theory}, Appl. Math. Sci., Springer, New York, 2007.
		
		\bibitem{BL} {\sc E. Bl{\aa}sten and H. Liu}, {\em On vanishing near corners of transmission eigenfunctions}, arXiv:1701.07957.
		
		\bibitem{BL1} {\sc E. Bl{\aa}sten and H. Liu}, {\em On corners scattering stably and stable shape determination by a single far-field pattern}, arXiv:1611.03647.
		
		\bibitem{BLLW} {\sc E. Bl{\aa}sten, X. Li, H. Liu AND Y. Wang}, {\em On vanishing and localizing near cusps of transmission eigenfunctions: a numerical study}, https://arxiv.org/abs/1704.01885.
		
		\bibitem{BPS}
		{\sc E. Bl{\aa}sten, L. P\"aiv\"arinta and J. Sylvester}, {\em Corners always scatter}, Comm.\ Math.\ Phys., {331} (2014), 725--753.
		
		\bibitem{CC} {\sc F. Cakoni and D. Colton}, {\em Qualitative Methods in Inverse Scattering Theory: An Introduction}, Springer-Verlag, Berlin, 2005.
		
		\bibitem{CC2} {\sc F. Cakoni and D. Colton}, {\em Qualitative Methods in Inverse Scattering Theory}, Springer-Verlag, Berlin, 2006.
		
		\bibitem{CCC} {\sc F. Cakoni, M. Cayoren and D. Colton}, {\em Transmission eigenvalues and the nondestructive testing of dielectrics}, Inverse Problems 24 (2008) 065016 (15pp).
		
		\bibitem{CCM} {\sc F. Cakoni, D. Colton and P. Monk}, {\em On the use of transmission eigenvalues to estimate the index of refraction from far field data}, Inverse Problems 23 (2007) 507-522.
		
		\bibitem{CCMbook} {\sc F. Cakoni, D. Colton, and P. Monk}, {\em The Linear Sampling Method in Inverse Electromagnetic Scattering}, SIAM, Philadelphia, 2011.
		
		\bibitem{CGH10}
		{\sc F. Cakoni, D. Gintides and H. Haddar},
		{\em The existence of an infinite discrete set of transmission eigenvalues},
		{SIAM J. Math. Anal., {\bf 42} (2010), 237--255}.
		
		\bibitem{CHad}
		{\sc F. Cakoni, H. Haddar}, {\em On the existence of transmission eigenvalues in an inhomogeneous medium}, Appl. Anal. 88:4 (2009), 475-493. MR 2010m:35557 Zbl 1168.35448.
		
		
		\bibitem{CKirs} {\sc D. Colton and A. Kirsch}, {\em A simple method for solving inverse scattering problems in the resonance region}, Inverse Problems, 12 (1996), pp. 383–393.
		
		\bibitem{CKP} {\sc D. Colton, A. Kirsch and L. Paivarinta}, {\em Far-field pattern for acoustic waves in an inhomogeneous medium}, SIAM J. Math. Anal. Vol. 20, No. 6, pp. 1472-1483, November 1989.
		
		\bibitem{CK} {\sc D. Colton and R. Kress}, {\em Inverse Acoustic and Electromagnetic Scattering Theory}, 2nd ed., Springer-Verlag, Berlin, 1998.
		
		\bibitem{CMS10}{\sc D. Colton, P. Monk and J. Sun}, {\em Analytical and computational methods for transmission eigenvalues},{Inverse Problems, {\bf 26} (2010), 045011}.
		
		\bibitem{CPS} {\sc D. Colton, P. Paivarinta and J. Sylvester}, {\em The interior transmission problem},  Inverse Problems Imag. 2007(1) 13–28.
		
		\bibitem{EH1} {\sc J. Elschner and G. Hu}, {\em Corners and edges always scatter}, Inverse Problems, 31 (2015), 015003, 1--17.
		
		\bibitem{EH2} {\sc J. Elschner and G. Hu}, {\em Acoustic scattering from corners, edges and circular cones}, arXiv: 1603.05186.
		
		\bibitem{HSV} {\sc G. Hu, M. Salo and E. Vesalainen}, {\em Shape identification in inverse medium scattering}, SIAM J. Math. Anal., 48 (2016), 152--165.
		
		
		\bibitem{JL} {\sc X. Ji and H. Liu}, {\em On isotropic cloaking and interior transmission eigenvalue problems}, arXiv:1604.05498.
		
		\bibitem{KG} {\sc A. Kirsch and N. Grinberg}, {\em The Factorization Method for Inverse Problems}, Oxford University Press, Oxford, 2008.
		
		\bibitem{LR} {\sc A. Lechleiter and M. Rennoch}, {\em Inside-outside duality and the determination of electromagnetic interior transmission eigenvalues}, SIAM Journal on Mathematical Analysis, 47(1):684-705, 2015.
		
		\bibitem{LLLW} {\sc X. Li, J. Li, H. Liu and Y. Wang}, {\em Electromagnetic interior transmission eigenvalue problem for inhomogeneous media containing obstacles and its applications to near cloaking}, IMA Journal of Applied Mathematics, 2017, to appear.
		
		\bibitem{LLZ} {\sc J. Li, H. Liu and J. Zou}, {\em Locating multiple multiscale acoustic scatters}, Multiscal Model. Simul. Vol. 12, No. 3, pp. 927–952.
		
		\bibitem{LWZ} {\sc H. Liu, Y. Wang and S. Zhong}, {\em Nearly non-scattering
			electromagnetic wave set and its application}, Zeitschrift fur Angewandte Mathematik und Physik (ZAMP), 68 (2017), 68-35.
		
		\bibitem{PSV} {\sc L. P\"aiv\"arinta, M. Salo and E. Vesalainen}, {\em Strictly convex corners scatter}, Rev. Mat. Iberoamericana, in press.
		
		\bibitem{RS} {\sc L. Rondi and M. Sini}, {\em Stable Determination of a Scattered Wave from its Far-Field Pattern: The High Frequency Asymptotics}, Arch. Rational Mech. Anal., 218 (2015),
		1–54.
		
		\bibitem{S} {\sc J. Sylvester}, {\em Transmission Eigenvalues in One Dimension}, Inverse Problems 29 104009.
		
		\bibitem{SS} {\sc B. D. Sleeman and D. C. Stocks}, {\em Interior transmission eigenvalues of a rectangle}, Inverse Problems 32 (2016) 025010 (15pp).
		
		\bibitem{U} {\sc G. Uhlmann, ed.}, {\em Inside Out: Inverse Problems and Applications}, Math. Sci. Res. Inst. Publ. 47, Cambridge University Press, Cambridge, UK, 2003.
		
		\bibitem{Weck}{\sc N. Weck}, {\em Approximation by Herglotz wave functions}, Math. Methods Appl. Sci., 27(2004), 155–162.
		
		
	\end{thebibliography}
\end{document}